\documentclass[aos]{imsart2}

\RequirePackage{amsthm,amsmath,amsfonts,amssymb}
\RequirePackage[authoryear]{natbib}
\RequirePackage[colorlinks,citecolor=blue,urlcolor=blue]{hyperref}
\RequirePackage{graphicx}

\startlocaldefs
\theoremstyle{plain}

\newtheorem{theorem}{Theorem}[section]
\newtheorem{proposition}[theorem]{Proposition}
\newtheorem{lemma}[theorem]{Lemma}
\theoremstyle{definition}
    
\newtheorem{assumption}{Assumption}

\newtheorem*{remark}{Remark}

\endlocaldefs

\newcommand{\one}{\mathbf{1}}
\newcommand{\zero}{\mathbf{0}}
\newcommand{\PR}{\mathbb{P}}

\newcommand{\R}{\mathbb{R}}
\newcommand{\E}{\mathbb{E}}

\newcommand{\xb}{\bar{X}}
\newcommand{\xbs}{\bar{x}}
\newcommand{\yb}{\bar{Y}}

\newcommand{\yt}{\tilde{Y}}
\newcommand{\ybh}{\hat{\bar{Y}}}
\newcommand{\vh}{\hat{V}}
\newcommand{\sh}{\hat{S}}
\newcommand{\vhn}{\hat{V}_{\mathrm{ney}}}
\newcommand{\vhns}{\vhn^\mathrm{S}}
\newcommand{\tauh}{\hat{\tau}}

\newcommand{\tauhm}{\hat{\tau}^{\mathrm{mid}}}
\newcommand{\sigt}{\tilde{\sigma}}

\newcommand{\bern}{\mathrm{Bern}}
\newcommand{\cre}{\mathrm{CRE}}
\newcommand{\piu}{\underline{\pi}}
\newcommand{\nb}{\bar{n}}

\DeclareMathOperator{\diag}{diag}
\DeclareMathOperator{\tr}{tr}
\DeclareMathOperator{\ent}{Ent}
\DeclareMathOperator{\var}{Var}

\newcommand\indep{\protect\mathpalette{\protect\independenT}{\perp}}
\def\independenT#1#2{\mathrel{\rlap{$#1#2$}\mkern2mu{#1#2}}}

\begin{document}

\begin{frontmatter}
\title{Randomization Inference with Concentration Inequalities}
\runtitle{Randomization Inference with Concentration Inequalities}

\begin{aug}
\author{\fnms{Tobias}~\snm{Freidling} \ead[label=e1]{tobias.freidling@epfl.ch}\orcid{0000-0003-0724-4297}}
\address{Institute of Mathematics, \'Ecole Polytechnique F\'ed\'erale de Lausanne \printead[presep={ ,\ }]{e1}}
\end{aug}

\begin{abstract}
Randomization or design-based inference is becoming an increasingly popular tool for analysing data from randomized experiments: It does not require modelling assumptions on the distribution of outcomes or covariates, and hypothesis testing and estimation are respectively valid and unbiased in finite samples. Yet, confidence intervals for the sample average treatment effect (SATE) are still constructed via finite-population central limit theorems and their coverage is only asymptotic. In this work, we explore an alternative approach: We use concentration inequalities to construct confidence intervals for the SATE with \emph{non-asymptotic} guarantees. We develop this approach for the most common experimental designs (Bernoulli trials and completely randomized experiments) and provide Hoeffding and Bernstein-type confidence intervals. Moreover, we extend these results to matched-pair, cluster and stratified randomized experiments. Our key technical contributions are a novel Bernstein-type concentration inequality for i.i.d.\ data points as well as a concentration result for Neyman's variance estimator.
\end{abstract}



\end{frontmatter}

\section{Introduction}

Randomized experiments are conducted in almost every area of applied statistics as they are the most credible method to elicit causal relationships. While there are many different methods tailored to the specific characteristics of particular experiments, most of them are based on a \emph{random sampling paradigm}:
The data points are assumed to be independent draws from an underlying, infinitely large super-population; alternatively, one may posit that there exists a data-generating distribution (or rather a statistical model) which the observations are sampled from. Both perspectives lead to the familiar assumption of i.i.d.\ data.

In some cases, the existence of such an infinite super-population or a data-generating mechanism may be justifiable or at least a good approximation; in other situations, however, it may be dubious. \cite{manski_how_2018} discuss such an example in the context of crime data in the USA: If the unit of interest is defined as a US citizen, assuming an infinite super-population can be a reasonable approximation to the true super-population of roughly 340 million Americans. Yet, if a unit is one of the 50 American states in a certain year, the existence of an infinitely large super-population of states or a probabilistic process that creates them seems far-fetched.


This fundamental issue can be circumvented by developing methods based on a \emph{random design paradigm} instead: We only consider units in the dataset -- hence a finite population, condition on their covariates and outcomes and only use the randomness in the assignment to different treatment arms of the study for inference. Since the treatment assignment distribution is known in a randomized experiment, we can not only avoid the assumption of an infinite super-population but also any, potentially wrong modeling choices. We refer to \citet{imbens_causal_2015,zhang_2023_randomization_test,ding_first_2024} for a review of randomization/design-based inference. While assumptions on the sampling mechanism from a (finite or infinite) super-population are still needed to assess external validity of the findings, inference for the observed units does \emph{not} require them. 


In randomized experiments, we are often interested in the sample average treatment effect (SATE) for the finite population of $n$ recruited units, which is defined as
\begin{equation*}
    \tau_n = \frac{1}{n} \sum_{i=1}^n Y_i(1)-Y_i(0).
\end{equation*}
Here, $Y_i(1)$ and $Y_i(0)$ describe the outcomes of the $i$-th unit under treatment and control, respectively. One of the most common treatment assignment mechanisms is the completely randomized experiment (CRE): We fix the size of the treated and control group a priori and then accordingly
choose $n_1$ of the $n$ recruited units at random to receive treatment ($Z_i=1$) and the remaining $n_0=n-n_1$ units to receive control ($Z_i=0$). Estimation under this and many other designs is comparatively straightforward as we can use the unbiased Horvitz-Thompson (HT) estimator
\begin{equation*}
    \tauh_n = \frac{1}{n_1} \sum_{i=1}^n Z_i Y_i(1)-\frac{1}{n_0}\sum_{i=1}^n(1-Z_i)Y_i(0).
\end{equation*}

Uncertainty quantification, on the other hand, is somewhat more involved. To construct confidence intervals for the SATE or related estimands, the predominant approach is based on finite population central limit theorems, see \citet{hajek_1960,li_general_2017} and references therein. The corresponding result in our setting states that under some regularity conditions
\begin{equation*}
    \frac{\tauh_n-\tau_n}{\sqrt{\var(\tauh_n)}} \,\to\, \mathcal{N}(0, 1),\quad \text{as}\quad n \to\infty.
\end{equation*}
The confidence interval that stems from inverting this central limit theorem only has asymptotic coverage guarantees which may be undesirable when we are interested in a finite population. Yet, more concerningly, when the assumption of an infinite super-population is not plausible, it is unclear if taking the limit $n\to \infty$ is sensible in the first place. This is also reflected in the fact that not only the estimator~$\tauh_n$ but also the estimand~$\tau_n$ depends on the recruited units.

In this article, we develop an alternative approach for constructing confidence intervals in the randomization inference framework. We derive novel concentration inequalities for the most common experimental designs and invert them to obtain confidence intervals for the SATE. These exhibit \emph{non-asymptotic} coverage and do not require taking the limit $n\to \infty$ or any notion of a super-population.







\subsection{Related Work} While the random sampling and design paradigm are philosophically quite different, the methods developed in the respective frameworks may be surprisingly similar. For instance, we can often use the same estimator for the variance in the finite-population CLT as in the familiar super-population CLT. \citet{imai_misunderstandings_2008} and \citet{ding_bridging_2017} elaborate on such differences and similarities. Moreover, a design-based approach can also incorporate an assumption on the sampling process from a finite super-population to address external validity. This idea is formulated and developed by \citet{miratrix_worth_2018}, \citet{abadie_samplingbased_2020} and \citet{yang_rejective_2023}, which treat mean estimation, regression analysis and re-randomization designs, respectively.

Randomization inference in causal inference is closely related to and shares overlapping origins with the fields of survey sampling and permutation tests. Survey sampling does typically not involve counterfactual variables and the associated impossibility of obtaining all data points of interest, but it may still be impractical to gather data from an entire (finite) population. Hence, one may recruit units at random to estimate a specific population quantity. Here, the random recruitment plays the role of the treatment assignment in a causal experiment; we refer to \citet{mercer_theory_2017} for a comparison of survey sampling and causal inference. While methods in survey sampling were predominantly developed under a design-based paradigm \citep{neyman_two_1934,cochran_sampling_1977}, model-based ideas have been influential, as well \citep{srndal_design-based_1978,sarndal_1992}. Permutation and randomization tests arose from the same strand of literature, e.g.\ \citet{fisher_design_1935} and \citet{pitman_significance_1937}, are in some cases computationally identical and can provide non-asymptotic guarantees. Hence, they are often not clearly distinguished from one another despite relying on quite different assumptions: Randomization inference is anchored in the \emph{design} of an experiment, whereas permutation tests invoke \emph{sampling} from an infinite super-population \citep{kempthorne_behaviour_1969,ernst_permutation_2004,lehmann_2006}. For a more in-depth discussion, we refer to \citet{zhang_2023_randomization_test}.

Concentration inequalities allow us to construct non-asymptotic tests and confidence intervals under minimal assumptions and have therefore come more into the spotlight of statistical research in recent years. In the context of surveys, results for sampling without replacement are of particular interest. In his seminal paper, \citet{serfling_probability_1974} showed that in this setting the usual Hoeffding inequality can indeed be improved -- a result which was subsequently refined by \citet{bardenet_concentration_2015} and generalized to more sampling schemes by \citet{bertail_bernstein-type_2019}. Concentration inequalities are also an integral tool in the growing literature on E-values and anytime-valid inference \citep{ramdas_game-theoretic_2023}. Here, the data is assumed to be observed sequentially and the goal is constructing a confidence sequence for a parameter of interest. \citet{waudby-smith_confidence_2020} treat the case of sequential sampling without replacement, whereas \citet{howard_time-uniform_2021} develop the larger theory in great generality. The resulting inference typically depends on the order of the data points. By contrast, our article considers the usual static setting and we derive confidence intervals that are invariant to this ordering.

Closest related to our work are the following articles. \citet{aronow_nonparametric_2025} and \citet{ding_what_2025} briefly discuss the idea of using a Hoeffding concentration inequality to construct a finite-sample confidence interval for the SATE, but do not develop this approach in detail. \citet{sandoval_nonasymptotic_2026} have recently introduced new Hoeffding-type concentration inequalities for the SATE under the completely randomized and Bernoulli design. The pertaining confidence intervals are particularly useful for small treatment assignment probabilities and provide a good benchmark for our results. \citet{li_exact_2016} propose non-asymptotic confidence intervals in a CRE when the outcomes are binary. In this restricted setting, we do not need to resort to concentration inequalities as the distribution of the HT estimator can actually be characterized. Lastly, \citet{shi_berryesseen_2026} derive Berry-Esseen bounds for the HT and related estimators which give further insight into their convergence behaviour to the normal distribution. While these results are still asymptotic, they are a step towards finite-sample guarantees and underline the interest in them.

\subsection{Organization of the Article} Section~\ref{sec:setting} introduces our notation, the class of $m$-centred estimators, which contains the standard Horvitz-Thompson estimator, as well as the assumptions that we make going forward. In Section~\ref{sec:bernoulli}, we investigate experiments with independent treatment assignments. We state and derive the relevant concentration inequalities, construct Hoeffding and Bernstein-type confidence intervals and extend our results to matched-pair designs. Section~\ref{sec:cre} deals with completely randomized treatment assignment schemes that fix the numbers of units assigned to the different arms. We derive Hoeffding and Bernstein confidence intervals and generalize our findings to cluster and stratified randomized experiments. Lastly, we discuss our results and elaborate on possible future research directions in Section~\ref{sec:discussion}. Most proofs and additional theorems are deferred to the appendix.

\section{Setting}\label{sec:setting}
We consider a finite population of $n$ units and adopt the framework of potential outcomes \citep{splawa-neyman_application_1990,rubin_estimating_1974}. According to this causal model, there exist two potential outcomes $Y_i(0)$ and $Y_i(1)$ for each unit $i \in [n]:= \{1,\ldots,n\}$ which correspond to the outcome under control and treatment, respectively. Here, we implicitly make the familiar assumptions of 'no interference', i.e.\ the outcomes of one unit do not depend on the treatment received by other units. Moreover, we may observe the covariate information~$L_i$ for the $i$-th unit. In the following, we use the notations $Y(\cdot) := (Y_i(0),Y_i(1))_{i\in[n]}$ and $L:=(L_i)_{i\in[n]}$.

The treatments assigned to the units in the experiment are denoted $Z := (Z_1,\ldots,Z_n) \in \{0,1\}^n$. In randomization inference, we condition on the potential outcomes and covariates and use the treatment assignment distribution $Z \mid Y(\cdot),L \sim \PR$ as the basis for inference. Going forward, we drop the conditioning on $Y(\cdot)$ and $L$ to keep the statements concise. The marginal treatment probability of the $i$-th unit is defined as $\pi_i := \E[Z_i]$ and we make the following assumption.
\begin{assumption}[Randomized Experiment]\label{ass:rand-exp} $Y(\cdot) \indep Z \mid L$, the treatment assignment distribution $\PR$ is known and $0 < \pi_i < 1$ for all $i\in[n]$.
\end{assumption}
Assumption~\ref{ass:rand-exp} contains
the usual (strong) exchangeability/ignorability and positivity conditions \citep{hernan2020causal,ding_first_2024} and was formulated in the same way by \citet{zhang_2023_randomization_test}. Note that the experimenter can choose the distribution $\PR$ and thus ensure that Assumption~\ref{ass:rand-exp} is satisfied. 

The estimand of interest is the sample average treatment effect (SATE) which is defined as $\tau_n := \frac{1}{n} \sum_{i=1}^n Y_i(1) - Y_i(0)$ and can be estimated by the classical \citet{horvitz_generalization_1952} estimator, abbreviated as HT estimator. In this article, we introduce the more general class of $m$-centred estimators that encompasses all
\begin{equation*}
    \tauh_n^{m} := \frac{1}{n} \sum_{i=1}^n \tauh^m_{n,i} := \frac{1}{n} \sum_{i=1}^n \big(Y_i(1)-m\big)\, \frac{Z_i}{\pi_i}-\big(Y_i(0)-m\big)\, \frac{1-Z_i}{1-\pi_i},
\end{equation*}
where $m \in\R$. It is a subset of the general class of unbiased SATE estimators described by \citet{aronow_class_2013} and contains the vanilla HT estimator as a special case ({$\tauh_n=\tauh_n^0$}). Of particular interest will be $\tauh^{(a+b)/2}_n$ which was named \emph{midpoint-differenced estimator} in a recent paper by \citet{aronow_minimax_2026} and will be abbreviated as $\tauhm_n$ in the following. All $m$-centred estimators are unbiased but may exhibit different concentration properties under some designs.

We can compute $\tauh^m_n$ from the collected data if we invoke the standard assumption that observed and potential outcomes are 'consistent'. That is, if unit $i \in [n]$ received treatment $Z_i$ and the outcome $Y_i$ was observed, then $Y_i(Z_i) = Y_i$. In this work, we focus on the Horvitz-Thompson-type estimators as they are both unbiased and linear in the treatment assignments which facilitates deriving concentration inequalities. The estimator proposed by \citet{hajek_1958,hajek_comment_1971} often has smaller variance, but is slightly biased and involves the ratio of two linear expressions in $Z$ which compounds constructing non-asymptotic confidence intervals.

Lastly, we assume that the potential outcomes are bounded in a real-valued interval $[a,b]$.
\begin{assumption}[Boundedness]\label{ass:bounded}
    $Y_i(0),Y_i(1) \in [a,b]$ for all $i \in [n]$.
\end{assumption}
This is the key condition that allows us to construct concentration inequalities and subsequently confidence intervals. Besides Assumptions~\ref{ass:rand-exp} and~\ref{ass:bounded}, no further prerequisites or regularity conditions are needed.

\section{Independent Assignment}\label{sec:bernoulli}
First, we focus on the case of independent treatment assignments, i.e.\ the random variables $Z_1,\ldots,Z_n$ are jointly independent (but not necessarily equally distributed). Such an experiment is known as a Bernoulli trial and denoted by $Z \sim \bern(\pi_1,\ldots,\pi_n)$, where the treatment assignment probabilities $\pi_i$ can be chosen based on the covariates. For this design, we can rely on classical concentration results. We start by reviewing them and introduce a new, slightly improved concentration inequality which may also be of independent interest. Then, we apply these tools to construct Hoeffding and Bernstein-type confidence intervals for Bernoulli trials and extend our results to matched-pair designs.

\subsection{Concentration Inequalities} We begin our review with Hoeffding's inequality and state the general version that allows the random variables to have different support.
\begin{theorem}[\cite{hoeffding_probability_1963}, Thm.~2]\label{thm:hoeffding}
    Let $X_1,\ldots,X_n$ be independent random variables which have respective means $\mu_1,\ldots,\mu_n$ and assume that $X_i \in [l_i,u_i]$ almost surely for all $i\in\{1,\ldots,n\}$. Let $\delta \in (0,1)$. Then, with probability $1-\delta$,
    \begin{equation*}
        \left\vert \frac{1}{n}\sum_{i=1}^n X_i-\mu_i \right\vert \leq \sqrt{\tfrac{1}{n}{\textstyle\sum_{i=1}^n  (u_i-l_i)^2}}\,\sqrt{\frac{\log(2/\delta)}{2n}}.
    \end{equation*}
\end{theorem}
This theorem is quite elegant and provides good bounds when the distributions of the random variables are rather dispersed. If they are concentrated, however, Bernstein inequalities \citep{bernstein1924,bernstein1946} can provide tighter bounds as they replace the length of the interval with a variance expression in the leading term. Since the variance is typically unknown, we have to estimate it and account for the additional uncertainty. Such results are consequently referred to as empirical Bernstein inequalities. Here, we state a widely used theorem by \citeauthor{maurer_empirical_2009} in a more general version, where the random variables can have different support.\footnote{The general version directly follows from applying \citeauthor{maurer_empirical_2009}'s original theorem to the transformed random variables $1/2+(X_i-c_i)/L \in [0,1]$.}


\begin{theorem}[\cite{maurer_empirical_2009}, Thm.~11]\label{thm:mp}
    Assume the setting of Theorem~\ref{thm:hoeffding} and denote the centred sample variance $\vh = \frac{1}{n(n-1)} \sum_{i<j} \big((X_i-c_i) - (X_j-c_j)\big)^2$, where $c_i := (l_i+u_i)/2$. Define $L:=\max_{1\leq i \leq n}(u_i-l_i)$ and let $\delta \in (0,1)$. 
    Then, with probability $1-\delta$,
    \begin{equation*}
        \left\vert \frac{1}{n}\sum_{i=1}^n X_i-\mu_i \right\vert \leq \sqrt{\frac{2 \vh \log(3/\delta)}{n}} + \frac{7L \log(3/\delta)}{3(n-1)}.
    \end{equation*}
\end{theorem}

This result arises from combining a Bernstein inequality with known variance with a separate concentration inequality for the estimator $\vh$ via the union bound. As the literature on self-normalizing processes \citep{de_la_pena_self-normalized_2009} demonstrates, the union bound argument can sometimes be improved by deriving a \emph{single} concentration inequality for a suitably normalized quantity. This idea has been adopted by the growing anytime-valid inference community in order to derive uniform martingale concentration inequalities, see e.g.\ \citet{howard_time-uniform_2021, waudby-smith_estimating_2024}. In the static setting that we consider, these inequalities remain valid but they depend on the order of the data points.

To remedy this undesirable feature, we employ a symmetrization technique recently introduced by \citet{barber_hoeffding_2024} and derive a concentration result that both improves \citet{maurer_empirical_2009}'s empirical Bernstein inequality and is invariant under re-ordering of the data points. We provide a tighter, implicit bound akin to Bennett's inequality and a looser, explicit bound akin to Bernstein's inequality. The proof is deferred to Appendix~\ref{app:self-normalizing-conc-inq}.
\begin{theorem}\label{thm:mine-bernstein}
    Assume the setting of Theorem~\ref{thm:mp}, define $\epsilon_n := \frac{1}{n} \sum_{j=1}^{n} \frac{1}{j} =\mathcal{O}(\frac{\log(n)}{n})$ and let $\delta \in (0,1)$. Then, with probability $1-\delta$,
    \begin{align*}
        \left\vert\frac{1}{n}\sum_{i=1}^n X_i - \mu_i\right\vert 
        &\leq (1+\epsilon_{n-1})\frac{\vh}{L}\, h_1^{-1}\!\left(\frac{L^2\log(2/\delta)}{(1+\epsilon_{n-1})(n-1)\vh}\right),\\
        &\leq\sqrt{\frac{2(1+\epsilon_{n-1})\,\vh \,\log(2/\delta)}{n-1}} + \frac{L\log(2/\delta)}{n-1},
    \end{align*}
    where $h_1(u) := u - \log(1+u)$.
\end{theorem}

\subsection{Hoeffding Confidence Intervals}
In a Bernoulli trial, we can directly apply Hoeffding's inequality to the $m$-centred estimator to obtain a confidence interval for the SATE.

\begin{proposition}\label{prop:bernoulli-hoeffding}
    Suppose $Z\sim \bern(\pi_1,\ldots,\pi_n)$, define $B:=\max\{\lvert a-m\rvert, \lvert b- m \rvert\}$ and let $\alpha \in (0,1)$. Then, a $1-\alpha$ confidence interval for $\tau_n$ is given by
    \begin{equation*}
        \left[\tauh_n^m \pm B \sqrt{\frac{1}{n}\sum_{i=1}^n\frac{1}{\pi_i^2(1-\pi_i)^2}}\, \sqrt{\frac{\log(2/\alpha)}{2n}}\right].
    \end{equation*}
    Among the class of $m$-centred estimators, $\tauhm_n$ has the shortest Hoeffding confidence interval, which is given by
    \begin{equation*}
        \left[\tauhm_n \pm \frac{b-a}{2} \sqrt{\frac{1}{n}\sum_{i=1}^n\frac{1}{\pi_i^2(1-\pi_i)^2}}\, \sqrt{\frac{\log(2/\alpha)}{2n}}\right].
    \end{equation*}
\end{proposition}

\begin{proof} The $m$-centred estimator is the average of the independent random variables $\tauh^m_{i,n}$
which can assume two values each. Hence, the length of the support of $\tauh^m_{i,n}$ is given by
\begin{equation*}
    u_i-l_i
        = \left\vert\frac{Y_i(1)-m}{\pi_i}+\frac{Y_i(0)-m}{1-\pi_i} \right\vert
        \leq \frac{\max\{\lvert a -m \rvert, \lvert b-m\rvert\}}{\pi_i(1-\pi_i)}.
\end{equation*}
    Applying Theorem~\ref{thm:hoeffding} yields the first statement of the theorem. Lastly, we see that the choice $m=\frac{a+b}{2}$ minimizes the constant $B$ and thus the length of the confidence interval.
\end{proof}

Proposition~\ref{prop:bernoulli-hoeffding} shows that the Hoeffding confidence interval around the classical HT estimator is not location invariant. That is, its length depends on $B=\max\{\lvert a\rvert, \lvert b\rvert\}$ which may be quite large even if the interval $[a,b]$ itself is short. Centring the potential outcomes is an easy solution to this problem and the midpoint-differenced estimator indeed has the shortest confidence interval.

In the special case of equal treatment assignment probabilities for all units,
\citet{sandoval_nonasymptotic_2026} have recently introduced a new confidence interval based on sub-Bernoulli concentration. They show that the length of their interval scales as $\mathcal{O}(1/\sqrt{\pi n})$ as $\pi \to 0$. This improves upon the $\mathcal{O}(1/\sqrt{\pi^2 n})$-rate of the Hoeffding confidence interval. One can generalize their result to $m$-centred estimators and improve the constants; here, we state an accordingly modified version of their theorem~3.5.

\begin{proposition}\label{prop:bern-sandoval}
    Suppose $Z\sim \bern(\pi_1,\ldots,\pi_n)$ and the treatment assignment probabilities are equal, i.e.\ $\pi:=\pi_1=\ldots=\pi_n$. Let $\alpha \in (0,1)$. Then, a $1-\alpha$ confidence interval for $\tau_n$ is given by
    \begin{equation*}
        \left[\tauh_n^m \pm  \frac{\log(2/\alpha)+n\gamma}{n\lambda}\right],
    \end{equation*}
    where $\gamma$ and $\lambda$ are defined as
    \begin{gather*}
        \gamma := \log\left(\frac{u}{u-l}e^{\lambda l} - \frac{l}{u-l}e^{\lambda u}\right),\qquad
        \lambda := \sqrt{\frac{2\log(2/\alpha)}{n(-l)u}},
    \end{gather*}
    with $l:=\min\{(a-m)/\pi, (m-b)/(1-\pi)\}$ and $u:= \max\{(b-m)/\pi,(m-a)/(1-\pi)\}$.
    
\end{proposition}

\begin{proof}
    In the proof of their theorem~3.5, \citet{sandoval_nonasymptotic_2026} consider the difference of an ($m$-centred) estimator and the SATE:
    \begin{align*}
        \tauh_n^m-\tau_n &= \frac{1}{n} \sum_{i=1}^n \left(\frac{Z_i}{\pi_i}-1\right)(Y_i(1)-m) - \left(\frac{1-Z_i}{1-\pi_i}-1\right)(Y_i(0)-m)\\
        &= \frac{1}{n} \sum_{i=1}^n \frac{(1-\pi_i)Y_i(1)+\pi_i Y_i(0) -m}{\pi_i(1-\pi_i)} (Z_i-\pi_i) =: \frac{1}{n} \sum_{i=1}^n X_i.
    \end{align*}
    We see that the random variables $X_i$ are upper- and lower-bounded as follows
    \begin{equation*}
        \min\left\{\frac{a-m}{\pi},\frac{m-b}{1-\pi}\right\}\leq X_i \leq \max\left\{\frac{b-m}{\pi},\frac{m-a}{1-\pi}\right\}.
    \end{equation*}
    We define $l$ as the minimum and $u$ as the maximum.\footnote{\citeauthor{sandoval_nonasymptotic_2026} use the notation $a$ and $b$ instead.} The result now follows from the same steps in their proof with the improved bounds on the $X_i$.
\end{proof}

In Figure~\ref{fig:bern-hoeffding}, we compare the length of the Hoeffding interval from Proposition~\ref{prop:bernoulli-hoeffding} with the confidence intervals proposed by \citet{sandoval_nonasymptotic_2026} from Proposition~\ref{prop:bern-sandoval}. We notice that the choice of the hyper-parameter for the latter intervals is more involved as it should also depend on the value $\pi$; by contrast, when we employ Hoeffding confidence intervals the midpoint-differenced estimator achieves the smallest interval length regardless of the treatment assignment probability. Moreover, we see that in the limit $\pi \to 0$ \citeauthor{sandoval_nonasymptotic_2026}'s intervals are shorter due to the improved scaling behaviour. For moderately large $\pi$, however, Hoeffding's inequality still yields tighter confidence intervals.

\begin{figure}[htbp]
    \centering
    \includegraphics[scale=0.6]{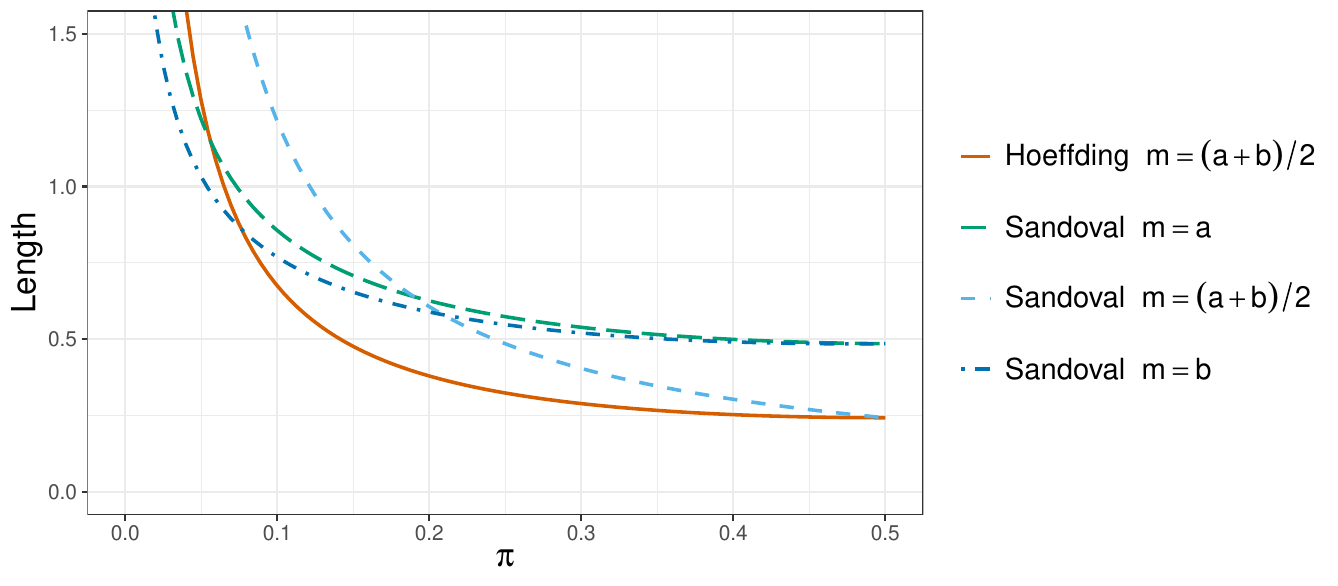}
    \caption{Length of Hoeffding-type confidence intervals in a Bernoulli trial as a function of $\pi$. We set $\alpha = 0.05, n=500$ and $[a,b] = [0,1]$ and only display values of $\pi$ smaller than 0.5 due to symmetry.}
    \label{fig:bern-hoeffding}
\end{figure}

\subsection{Bernstein-type Intervals} Next, we derive confidence intervals that incorporate an empirical estimate of the variance. To this end, we can apply Theorems~\ref{thm:mp} and~\ref{thm:mine-bernstein}. Here, we only use our sharpened concentration inequality and state the result for the midpoint-differenced estimator. The proposition for general $m$-centred estimators, the confidence interval based on \citet{maurer_empirical_2009}'s concentration inequality as well as the pertaining proof can be found in Appendix~\ref{app:bernoulli-general-bernstein}.

\begin{proposition}\label{prop:confint-bernstein-indep}
    Suppose $Z\sim \bern(\pi_1,\ldots,\pi_n)$, $n\geq 2$ and let $\alpha \in (0,1)$. Define $\piu := \min_{1 \leq i \leq n} \{\pi_i, 1-\pi_i\}$, $\vh := \frac{1}{n-1}\sum_{i=1}^n (\tauhm_{i,n}-\tauhm)^2$ and $\epsilon_n := \frac{1}{n} \sum_{j=1}^{n} \frac{1}{j}$. Then, an explicit $1-\alpha$ confidence interval for~$\tau_n$ is given by
    \begin{gather*}
        \left[\tauhm_n \pm \left(\sqrt{\frac{2(1+\epsilon_{n-1})\vh\log(2/\alpha)}{n-1}} + \frac{b-a}{\piu}\frac{\log(2/\alpha)}{n-1}\right)\right].
    \end{gather*}
    The tighter implicit $1-\alpha$ confidence interval takes the form
    \begin{equation*}
        \left[\tauhm_n \pm (1+\epsilon_{n-1})\frac{\piu}{b-a} \vh\, h_1^{-1}\!\left(\frac{(b-a)^2\log(2/\alpha)}{\piu^2(1+\epsilon_{n-1})(n-1)\vh}\right)\right].
    \end{equation*}
\end{proposition}

The first confidence interval in Proposition~\ref{prop:confint-bernstein-indep} can be easily compared to the standard Wald confidence interval based on a super-population central limit theorem: $[\tauhm \pm \Phi^{-1}(1-\alpha/2) \,\vh^{1/2}/\sqrt{n}\,]$. Note that the latter interval is derived under the random sampling paradigm and only has asymptotic coverage. For large $n$, the ratio of the lengths of the two intervals is approximately equal to
\begin{equation*}
    \frac{\sqrt{2\log(2/\alpha)}}{\Phi^{-1}(1-\alpha/2)}.
\end{equation*}
For $\alpha=0.05$, this expression evaluates to $\approx 1.39$. That is, compared to the standard Wald interval, our Bernstein-type interval is approximately 39\% wider, which improves upon the interval based on \citeauthor{maurer_empirical_2009}'s inequality which achieves 46\%. 

In Figure~\ref{fig:bern-bernstein}, we display the length of different confidence intervals as a function of the sample size $n$. We observe that Bernstein-type intervals can be substantially shorter than the Hoeffding interval, even when the outcomes are only moderately concentrated. Furthermore, for small $n$, the lower order $\mathcal{O}(n^{-1})$-term in the Bernstein-type intervals is palpable which explains the larger gap between \citeauthor{maurer_empirical_2009}'s confidence interval and ours. Lastly, the length of our explicit and implicit confidence are quite similar and asymptotically the same.

\begin{figure}[htbp]
    \centering
    \includegraphics[scale=0.6]{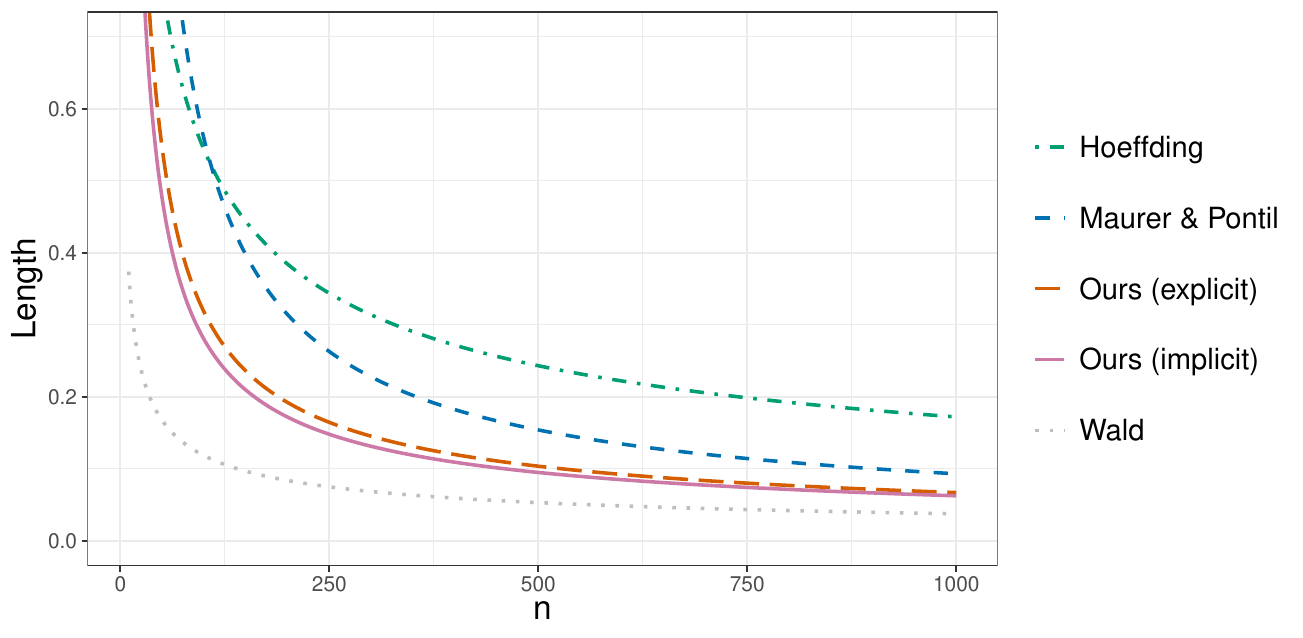}
    \caption{Length of confidence intervals in a Bernoulli trial as a function of $n$. We set $\alpha = 0.05, \pi=1/2$ and $[a,b] = [0,1]$. Moreover, we use $\vh = \sigma^2_{5,5}/(\pi(1-\pi))$, where $\sigma^2_{5,5}$ is the variance of a $\mathrm{Beta}(5,5)$-distribution.}
    \label{fig:bern-bernstein}
\end{figure}

\subsection{Matched-Pair Designs}\label{sec:matched-pairs} In order to increase covariate balance between the two treatment groups, oftentimes a matching procedure is used to find pairs of units with similar features and then randomize one unit to treatment and the other to control. In this subsection, we assume that the number of units $n$ is even and that, without loss of generality, the units are ordered with respect to the matching, i.e.\ matched units are indexed by one of the pairs $(2i-1, 2i)$, where $i\in \{1,\ldots,n/2\}$. The treatment assignment mechanism is denoted $Z \sim \mathrm{MP}(n)$ and can be formalized as follows: For independent $Z^*_1,\ldots,Z^*_{n/2} \sim \bern(0.5)$, we set $Z_{2i-1} = Z^*_i$ and $Z_{2i} = 1-Z^*_i$ for all $i\in\{1, \ldots, n/2\}$.

Since the matching procedure is only based on the covariate information, we do not need to account for it in design-based inference. Hence, practitioners can freely choose their preferred algorithm and the ($m$-centred) HT estimator remains unbiased. To derive confidence intervals, it is convenient to re-formulate it in terms of the $Z^*_i$ variables:
\begin{align*}
    \tauh^m_n &= \frac{1}{n} \sum_{i=1}^n \big(Y_i(1)-m\big)\, \frac{Z_i}{\pi_i}-\big(Y_i(0)-m\big)\, \frac{1-Z_i}{1-\pi_i}\\
    &= \frac{1}{n} \sum_{i=1}^{n/2} \big(Y_{2i-1}(1)-Y_{2i}(0)\big) \frac{Z_{2i-1}}{\pi_{2i-1}} - \big(Y_{2i-1}(0)-Y_{2i}(1)\big)\frac{1-Z_{2i-1}}{1-\pi_{2i-1}}\\
    &= \frac{1}{n/2} \sum_{i=1}^{n/2} \yt_i(1) \frac{Z^*_i}{0.5}-\yt_i(0)\frac{1-Z^*_{i}}{1-0.5} =: \frac{1}{n/2}\sum_{i=1}^{n/2} \tauh^*_{i,n},
\end{align*}
where the pseudo-potential outcomes $\yt_i(0)$ and $\yt_i(1)$ are defined as
\begin{equation*}
    Y^*_i(0) := \frac{1}{2} \big(Y_{2i-1}(0)-Y_{2i}(1)\big),\qquad
    Y^*_i(1) := \frac{1}{2}\big(Y_{2i-1}(1)-Y_{2i}(0)\big),
\end{equation*}
for $i \in\{1,\ldots,n/2\}$. All of them are contained in the interval $[-\frac{b-a}{2}, \frac{b-a}{2}]$; consequently, for any $m$, $\tauh^m_n$ is the midpoint-differenced estimator for the pseudo-outcomes and we can drop the superscript $m$. We can now directly apply Propositions~\ref{prop:bernoulli-hoeffding} and~\ref{prop:confint-bernstein-indep} to obtain confidence intervals. For brevity, we only state the Hoeffding and our explicit Bernstein interval here. 
\begin{proposition}
    Suppose $Z \sim \mathrm{MP}(n)$, $n\geq 4$ and let $\alpha \in (0,1)$. Then, a $1-\alpha$ Hoeffding confidence interval for $\tau_n$ is given by
    \begin{equation*}
        \left[\tauh_n \pm 2 (b-a)\sqrt{\frac{\log(2/\alpha)}{n}}\right].
    \end{equation*}
    Let $\hat{V}^*:= \frac{1}{n/2-1}\sum_{i=1}^{n/2} (\hat{\tau}^*_{i,n}-\tauh_n)^2$ and $\epsilon_n := \frac{1}{n} \sum_{j=1}^{n} \frac{1}{j}$. An explicit $1-\alpha$ Bernstein-type confidence interval is given by
    \begin{equation*}
        \left[\tauh_n \pm \left(\sqrt{\frac{2(1+\epsilon_{n/2-1})\hat{V}^*\log(2/\alpha)}{n/2-1}} + 2(b-a)\frac{\log(2/\alpha)}{n/2-1}\right)\right].
    \end{equation*}
\end{proposition}

\section{Completely Randomized Assignment}\label{sec:cre} Next, we consider experiments that randomly assign a pre-specified number of $n_1$ units to receive treatment and $n_0 = n-n_1$ units to receive control. This design is known as the completely randomized experiment (CRE), its treatment assignment distribution is given by $\PR(Z=z) = 1\big/\binom{n}{n_1}$, where $\sum_{i=1}^n z_i = n_1$, and dubbed $Z \sim \cre(n_0,n_1)$. We note that the treatment assignment probability for every unit is $\pi = n_1/n$ and that all $m$-centred estimators are equal under this design. Hence, we drop the superscript $m$ and state the results for the HT estimator going forward.

The key complication in constructing confidence intervals is the dependence between the different treatment assignments $Z_i$; fewer concentration inequalities are known for this setting as result. Notably, \citet{serfling_probability_1974} derived inequalities for sampling without replacement which were later improved by \citet{bardenet_concentration_2015}, and \citet{barber_hoeffding_2024} derives concentration inequalities for weighted sums of exchangeable random variables. We can recast the HT estimator and our treatment assignment distribution in their respective settings. In the following, we derive Hoeffding confidence intervals, prove a novel concentration inequality for Neyman's variance estimator and construct the pertaining empirical Bernstein confidence intervals. Lastly, we extend our results to cluster and stratified randomized experiments.

\subsection{Hoeffding Confidence Intervals}\label{sec:cre-hoeffding} Due to the dependence between treatment assignments for different units, the standard Hoeffding inequality of Theorem~\ref{thm:hoeffding} does not apply. Instead, we use a concentration inequality for weighted sums of exchangeable random variables from \citet{barber_hoeffding_2024}, which is stated in a slightly generalized version in Appendix~\ref{app:extension-literature}.
Since the $Z_i$'s are indeed exchangeable, this result is applicable and we can derive a Barber-Hoeffding confidence interval for the SATE.


\begin{proposition}\label{prop:cre-hoeffding}
    Suppose $Z\sim \cre(n_0,n_1)$, define $\epsilon'_n := \frac{\epsilon_n-1}{n-\epsilon_n} = \mathcal{O}(\frac{\log(n)}{n})$, where $\epsilon_n = \sum_{j=1}^n \frac{1}{j}$ and let $\alpha \in (0,1)$. Then, a $1-\alpha$ confidence interval for $\tau_n$ is given by
    \begin{equation*}
        \left[\tauh_n \pm \frac{b-a}{2} \frac{1}{\pi(1-\pi)}\, \sqrt{\frac{(1+\epsilon_n')\log(2/\alpha)}{2n}}\right].
    \end{equation*}
\end{proposition}
\begin{proof}
    Like in the proof of Proposition~\ref{prop:bern-sandoval}, we can express the difference of an $m$-centred estimator and the SATE as follows
    \begin{equation*}
        \tauh_n^m-\tau_n = \sum_{i=1}^n \left(\frac{1}{n}\frac{(1-\pi)Y_i(1)+\pi Y_i(0) -m}{\pi(1-\pi)}\right) (Z_i-\pi) =: \sum_{i=1}^n w_i X_i.
    \end{equation*}
    All $m$-centred estimators are equal to the HT estimator $\tauh$, but inserting the additional hyper-parameter helps us to tighten the constant of the interval and make it scale-invariant. Since $X_i \in \{-\pi,1-\pi\}$, the $X_i$ are supported on an interval of length $1$; moreover, $\bar{X}=0$. The norm of the weights can be upper bounded like
    \begin{equation*}
        \lVert w \rVert_2^2 \leq \frac{1}{n^2} \frac{1}{\pi^2(1-\pi)^2} n B^2,
    \end{equation*}
    where $B=\max\{\lvert a-m\rvert, \lvert b- m \rvert\}$. Choosing $B = (a+b)/2$ minimizes the upper bound and we obtain, $\lVert w \rVert_2 \leq \frac{b-a}{2}\frac{1}{\pi(1-\pi)}\frac{1}{\sqrt{n}}$. The confidence interval now directly follows from applying Theorem~\ref{prop:rina-hoeffding} with $N=n$.
\end{proof}

\citet{sandoval_nonasymptotic_2026} have recently derived an alternative Hoeffding confidence interval whose length scales as $\mathcal{O}(1/\sqrt{\pi n})$ instead of $\mathcal{O}(1/\sqrt{\pi^2 n})$ as $\pi \to 0$. To achieve this, they introduce `mini-batch complete randomization' which is marginally equivalent to the complete randomization procedure but allows them to keep track of groups of units. To accommodate this additional group structure, they need to slightly modify their HT estimator. In the special case that $\pi = 1/K$ for an integer $K\geq 2$ (or equivalently $n_0 = K'n_1$ where {$K'\in\mathbb{N}$}), this modified version is equal to the standard HT estimator. For brevity and comparability with Proposition~\ref{prop:cre-hoeffding}, we only state the confidence interval in this situation; for more details, the reader is referred to \citeauthor{sandoval_nonasymptotic_2026}'s article.

\begin{proposition}[\cite{sandoval_nonasymptotic_2026}, Thm.~3.1]\label{prop:cre-sandoval}
    Suppose $Z\sim \cre(n_0,n_1)$ and $\pi = 1/K$ for an integer $K \geq 2$. Let $\alpha \in (0,1)$. Then, a $1-\alpha$ confidence interval for $\tau_n$ is given by
    \begin{equation*}\label{eq:cre-sandoval}
        \left[\tauh_n \pm \frac{b-a}{\sqrt{\pi}}\sqrt{\frac{2\log(2/\alpha)}{n}}\right].
    \end{equation*}
\end{proposition}

Lastly, in a commentary on \cite{rigdon_randomization_2015}'s work, \cite{li_exact_2016} consider the special setting of binary outcomes. This additional structure allows them to characterize the distribution of the HT estimator in terms of hypergeometric distributions and compute confidence intervals by inverting randomization tests.

Figure~\ref{fig:cre-hoeffding} compares the length of the confidence intervals from Propositions~\ref{prop:cre-hoeffding} and~\ref{prop:cre-sandoval} as well as~\cite{li_exact_2016}. (To aid readability, we tacitly ignore that only finitely many values of $\pi$ are possible for finite $n$ and that the interval of Proposition~\ref{prop:cre-sandoval} only applies for $\pi = 1/K$.) Similar to the findings in Bernoulli trials, we observe that the better scaling behaviour of \citeauthor{sandoval_nonasymptotic_2026}'s intervals yields shorter lengths in the limit $\pi \to 0$. For moderately large $\pi$, however, Barber-Hoeffding intervals are preferable. For the case of binary outcomes, we clearly see that \cite{li_exact_2016}'s interval is the shortest. Nonetheless, if the sizes of the treatment and control group do not substantially differ, our Hoeffding interval gets reasonably close to this benchmark.


\begin{figure}[htbp]
    \centering
    \includegraphics[scale=0.55]{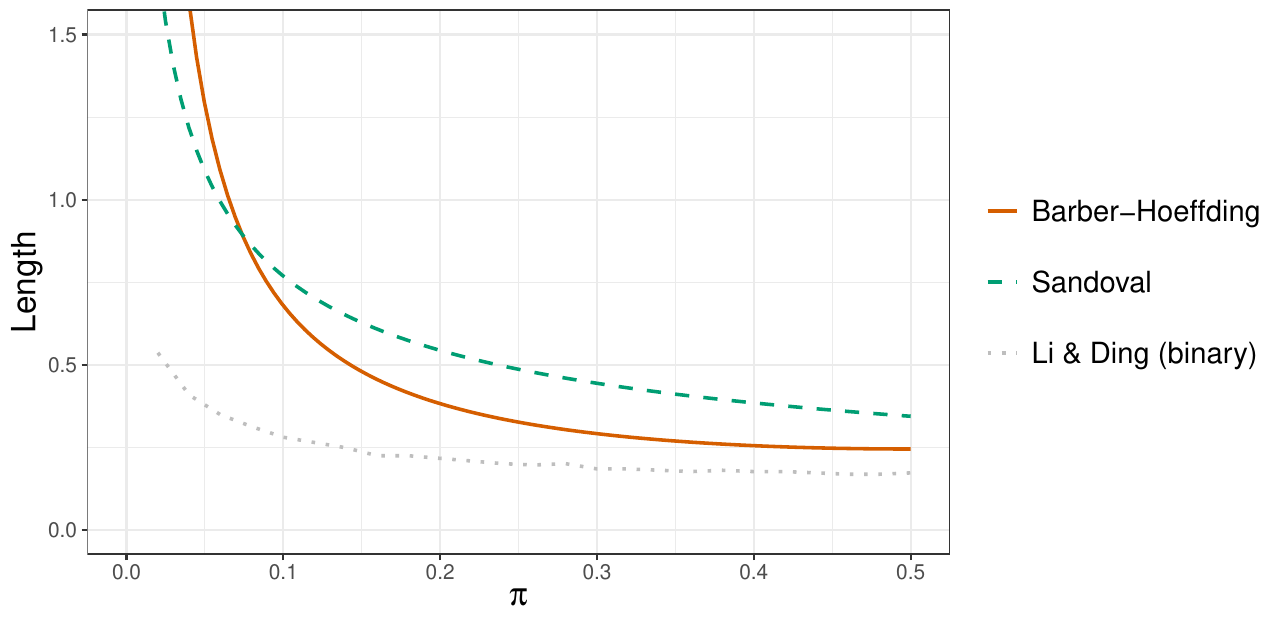}
    \caption{Length of Hoeffding-type confidence intervals in a completely randomized experiment as a function of $\pi$. We set $\alpha = 0.05, n=500$ and $[a,b] = [0,1]$ and only display values of $\pi$ smaller than 0.5 due to symmetry. (For \citet{li_exact_2016}'s confidence interval, we assume that $Y=1$ and $Y=0$ outcomes occur equally often in both treatment and control group and we use 2000 permutations to compute the randomization distribution in the inversion of tests.)}
    \label{fig:cre-hoeffding}
\end{figure}

\subsection{Bernstein Confidence Intervals}l\label{sec:bernstein} Next, we construct non-asymptotic confidence intervals that also take the variance of the HT estimator into account. In his celebrated \citeyear{splawa-neyman_application_1990}-paper, \citeauthor{splawa-neyman_application_1990} derived the decomposition
\begin{equation*}
	V:= \var(\sqrt{n}\,\tauh_n) = \frac{S(1)}{n_1/n} + \frac{S(0)}{n_0/n} - S(\tau),
\end{equation*}
where $S(0)$ and $S(1)$ are the the sample variances of the potential outcomes $Y_i(0)$ and $Y_i(1)$, respectively, and $S(\tau)$ is the sample variance of the individual treatment effects. (Note that we re-scale the variance so that $V=\mathcal{O}(1)$ in this work.)
Since all involved quantities are non-negative, but only the former two can be estimated, a commonly used conservative estimator of $V$ is given by
\begin{equation}\label{eq:var-ney}
	\vhn := \frac{\sh(1)}{n_1/n} + \frac{\sh(0)}{n_0/n},
\end{equation}
where we define
\begin{equation*}
	\sh(k) := \frac{1}{n_k-1} \sum_{i=1}^n \one\{Z_i=k\} (Y_i(k) - \ybh(k))^2,\qquad
	\ybh(k) := \frac{1}{n_k} \sum_{i=1}^n \one\{Z_i=k\} Y_i(k),
\end{equation*}
for $k \in \{0,1\}$. This conservative estimator is also used for asymptotic Wald-type intervals $[\tauh_n \pm \Phi^{-1}(1-\alpha/2)\vhn^{1/2}/\sqrt{n}]$, which are based on a finite-population central limit theorem \citep{li_general_2017}.

Since we did not succeed in deriving a concentration inequality involving $\vhn$ via a self-normalizing construction as in Section~\ref{sec:bernoulli}, we resort to the classical approach: We first prove a concentration inequality for $\vhn$ and then combine this result with an (oracle) Bernstein inequality via the union bound. The following theorem is one of our key contributions; we defer the full proof to Appendix~\ref{app:proof-neyman} but provide a sketch of the main steps here.

\begin{theorem}\label{thm:neyman-conc}
 Suppose $Z \sim \cre(n_0, n_1)$, where $n_0,n_1 \geq 2$, and let $\delta \in (0,1)$. Then, with probability $1-\delta$,
\begin{equation*}
    V^{1/2} \leq \E[\vhn]^{1/2} \leq \vhn^{1/2} + (b-a) \sqrt{\frac{c(n_0,n_1)\log(1/\delta)}{n}},
\end{equation*}
where the constant is given by
\begin{equation*}
    c(n_0,n_1) := \frac{n\,(n_0^2+n_1^2)^2}{n_0n_1^3(n_0-1) + n_1 n_0^3(n_1-1)}.
\end{equation*}
\end{theorem}

\begin{proof}[Proof sketch]
	First, in Lemma~\ref{lem:self-bounded}, we prove that Neyman's variance estimator has a self-boundedness property with respect to the swap operator, which permutes the treatment assignments of two units.
    
    Second, we realize that the completely randomized assignment mechanism is the stationary distribution of the Markov chain pertaining to the $n_1$-particle Bernoulli-Laplace diffusion model on $n$ sites. For these processes modified log-Sobolev concentration inequalities have been studied in the literature \citep{lee_logarithmic_1998} and also been recently linked to sampling without replacement \citep{sambale_concentration_2022}. In Theorem~\ref{thm:mod-log-sob}, we re-state the sharp inequality of \citet{salez_sharp_2021} in our setting.
    
    Third, we combine this concentration inequality with the self-boundedness property of~$\vhn$, and use an adapted version of Herbst's argument \citep{boucheron_concentration_2013}[chap.\ 6] to complete the proof.
\end{proof}

We are now in position to derive empirical Bernstein and Bennett confidence intervals. \citet{barber_hoeffding_2024} has recently proved an (oracle) Bernstein inequality for weighted sums of exchangeable random variables, which is amenable to our setting since the the treatments $Z_i$ are exchangeable and (functions of) the potential outcomes act as weights. In Theorem~\ref{thm:rina-bernstein} in the appendix, we state her theorem~3.3 and also derive the slightly tighter, analogue Bennett concentration inequality. Combining this theorem with the concentration of the Neyman estimator, cf.\ Theorem~\ref{thm:neyman-conc}, we get the following result. The proof is deferred to Appendix~\ref{app:cre-bernstein-barber}.

\begin{proposition}\label{prop:cre-bernstein}
    Suppose $Z \sim \cre(n_0,n_1)$, where $n_0,n_1 \geq 2$, and let $\alpha \in (0,1)$. Let $c(n_0,n_1)$ be defined as in Theorem~\ref{thm:neyman-conc}, and set $\epsilon'_n := \frac{\epsilon_n-1}{n-\epsilon_n}$ and $\epsilon_n'' := 4\epsilon_n'\frac{n(n-1)}{n_0n_1}-\frac{1}{n}$, where $\epsilon_n = \sum_{j=1}^n \frac{1}{j}$. Then, an explicit Bernstein $1-\alpha$ confidence interval for $\tau_n$ is given by
    \begin{equation*}
        \left[\tauh_n \pm \left(\sqrt{\frac{2(1+\epsilon_n'')(1+\epsilon_n') \vhn \log(3/\alpha)}{n}} + \frac{C(n_0,n_1)(b-a)}{n} \log(3/\alpha)\right)\right],
    \end{equation*}
    where
    \begin{equation*}
        C(n_0,n_1)= \frac{1+\epsilon_n'}{3}\frac{n(n-1)}{n_0\,n_1} + \sqrt{2(1+\epsilon_n'')(1+\epsilon_n')c(n_0,n_1)}.
    \end{equation*}
    The smaller implicit $1-\alpha$ Bennett confidence interval takes the form
    \begin{equation*}
        \left[\tauh_n \pm \frac{(1+\epsilon_n'')\,n_0\,n_1\,\tilde{V}}{n(n-1)(b-a)}\,h_2^{-1}\left(\frac{(1+\epsilon_n')\,n(n-1)^2\,(b-a)^2 \log(3/\alpha)}{(1+\epsilon_n'')\,n_0^2n_1^2\,\tilde{V}}\right)\right],
    \end{equation*}
    where $h_2(u):=(1+u)\log(1+u)-u$ and
    \begin{equation*}
        \tilde{V} = \left(\vhn^{1/2}+(b-a)\sqrt{\frac{c(n_0,n_1)\log(3/\alpha)}{n}}\right)^2.
    \end{equation*}
\end{proposition}
\begin{remark} Instead of basing our proof on \citet{barber_hoeffding_2024}'s Bernstein inequality, we could have also used \citet{bardenet_concentration_2015}'s concentration inequality for sampling without replacement. This generally leads to longer intervals due to the larger constant in the $\mathcal{O}(n^{-1})$-term, however. For the interested reader, we provide the corresponding proposition and a comparison with the confidence intervals of Proposition~\ref{prop:cre-bernstein} in Appendix~\ref{app:cre-bernstein-bardenet}.

\citet{greene_exponential_2017} also prove a Bernstein inequality for sampling from a finite population without replacement. Yet, their result uses the variance of a majorizing population which may be considerably larger than the actual variance. For this reason, we did not pursue a confidence interval based on their concentration inequality in this article.
\end{remark}

Analogously to the Bernoulli trial, we can compare the length of the Bernstein confidence interval of Proposition~\ref{prop:cre-bernstein} with the asymptotic Wald confidence interval $[\tauh_n\pm\Phi^{-1}(1-\alpha/2)\vhn^{1/2}/\sqrt{n}]$. For large $n$, the ratio of the lengths of the two intervals approaches
\begin{equation*}
    \frac{\sqrt{2\log(3/\alpha)}}{\Phi^{-1}(1-\alpha/2)},
\end{equation*}
which evaluates to $\approx 1.46$ for $\alpha = 0.05$. In finite samples, we find that the difference in length may be considerably bigger, though.

In Figure~\ref{fig:cre-bernstein}, we compare the length of the confidence intervals in Propositions~\ref{prop:cre-hoeffding} and~\ref{prop:cre-bernstein} as well as the Wald confidence interval as a function of the sample size $n$. We again observe that Bernstein- and Bennett-intervals can substantially improve upon Hoeffding intervals when the variance term is moderately small. Interestingly, the gap between the Bernstein and Bennett intervals is almost zero even for small $n$. Lastly, we notice that the Bernstein interval is rather wide compared to the Wald interval: While the ratio of their lengths is about 1.46 as $n\to \infty$, it still equals $\approx 2.47$ at $n=1000$ in the setting of the Figure~\ref{fig:cre-bernstein}. This slower ``convergence behaviour'' compared to Bernoulli trials can mostly be attributed to a larger constant in the lower order term of the Bernstein interval.

\begin{figure}[htbp]
    \centering
    \includegraphics[scale=0.65]{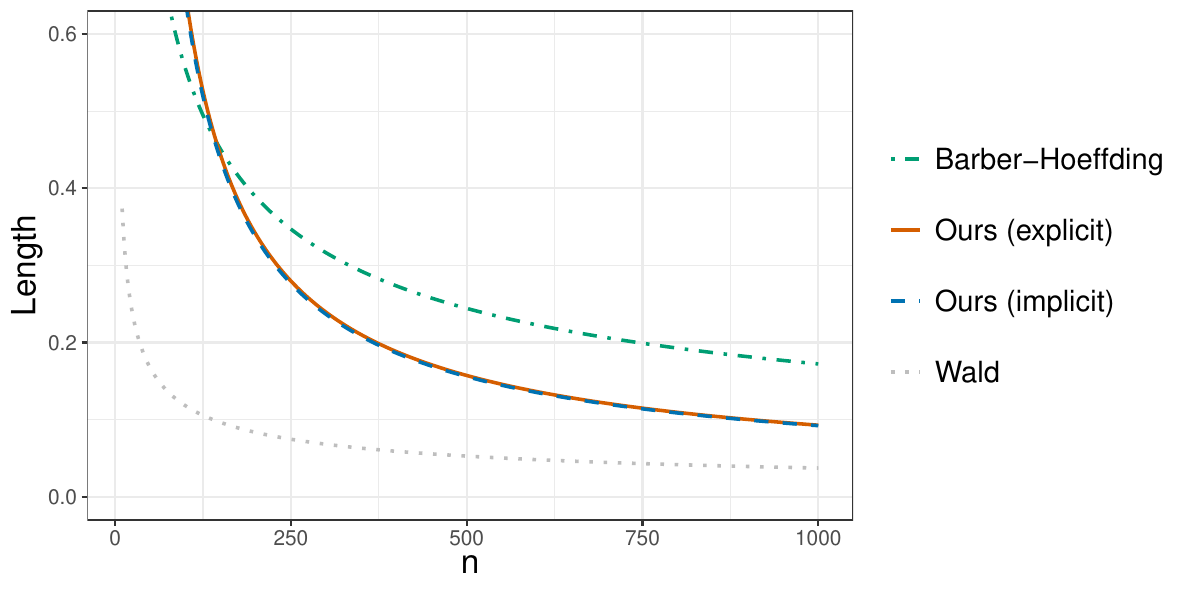}
    \caption{Length of confidence intervals in a completely randomized experiment as a function of $n$. We set $\alpha = 0.05, \pi=1/2$ and $[a,b] = [0,1]$. Moreover, we use $\vh = \sigma^2_{5,5}/(\pi(1-\pi))$, where $\sigma^2_{5,5}$ is the variance of a $\mathrm{Beta}(5,5)$-distribution.}
    \label{fig:cre-bernstein}
\end{figure}

\subsection{Cluster Randomized Experiments} In some settings, it is logistically difficult or unethical to assign each unit individually to treatment or control, but it may still be feasible to randomize groups of units. Depending on the application, these groups may be schools, households, hospitals or villages for instance. A study which randomizes at the group- instead of the individual-level is accordingly called
a cluster randomized experiment; for details about the design-based analysis of such trials we refer the reader to \citet{middleton_unbiased_2015,su_model-assisted_2021,schochet_design-based_2022}. Here, we consider the most popular version where groups/clusters enter a CRE.

Let $N$ be the number of groups and $n_j$ is the number of units in the $j$-th cluster, where $j\in\{1,\ldots,N\}$; consequently $n=\sum_{j=1}^N n_j$ is the total number of units in the experiment. The treatment assignment mechanism is denoted $Z \sim \mathrm{CCRE}(N_0,N_1; n_1,\ldots,n_N)$~and (assuming that the units are ordered according to their group membership)  defined as $Z = (Z^*_1 \one_{n_1}, \ldots,Z_N^*\one_{n_N})$, where $Z^* \sim \cre(N_0,N_1)$ and $\one_{n_j}:=(1,\ldots,1)\in \R^{n_j}$. Hence, all members of one group are administered the same treatment.

In order to ease exposition, we introduce double indexing and let $Z_{ji}$ and $Y_{ji}(0)$ and $Y_{ji}(1)$ denote the treatment and potential outcomes of the $i$-th unit in the $j$-th cluster, respectively, where $i\in\{1,\ldots,n_j\}$ and $j \in \{1,\ldots,N\}$. This allows us to re-formulate the $m$-centred HT estimator as follows:
\begin{align*}
    \tauh_n^m &= \frac{1}{n}\sum_{j=1}^N \sum_{i=1}^{n_j}(Y_{ji}(1) -m)\frac{Z_{ji}}{N_1/N} - (Y_{ji}(0) -m)\frac{1-Z_{ji}}{N_0/N}\\
    &= \frac{1}{N}\sum_{j=1}^N\left(\frac{N}{n}\sum_{j=1}^{n_j}(Y_{ji}(1)-m)\right) \frac{Z^*_j}{N_1/N} - \left(\frac{N}{n}\sum_{j=1}^{n_j}(Y_{ji}(0)-m)\right) \frac{1-Z^*_j}{N_0/N}\\
    &=: \frac{1}{N}\sum_{j=1}^N Y_j^*(1) \frac{Z^*_j}{N_1/N} - Y_j^*(0) \frac{1-Z^*_j}{N_0/N}.
\end{align*}
This shows that we can also regard a CCRE as a completely randomized experiment with the pseudo potential outcomes $Y_j^*(0)$ and $Y_j^*(1)$. Hence, the results from the previous section extend to this design, as well. Interestingly, though, while all $m$-centred HT estimators are unbiased, for unequal group sizes they are not equal. This property is different from CREs and requires choosing the hyper-parameter $m$ to minimize the length of the confidence intervals. As for independent treatment assignments, this is achieved at the midpoint $m=\frac{b-a}{2}$. Noticing that each $Y^*_j(k)$, $k\in\{0,1\}$, is supported on an interval of length $\frac{N}{n} (b-a)n_j \leq \frac{N}{n} (b-a) \max_{1\leq j\leq N} n_j$, we can directly apply Propositions~\ref{prop:cre-hoeffding} and~\ref{prop:cre-bernstein} to obtain confidence intervals. For brevity, we only state the Hoeffding and our explicit Bernstein interval here.

\begin{proposition}
    Suppose $Z\sim \mathrm{CCRE}(N_0, N_1; n_1,\ldots,n_N)$, where $N_0, N_1 \geq 2$ and let $\alpha \in(0,1)$. Let $\epsilon_N'$ be defined as in Proposition~\ref{prop:cre-hoeffding} and set $\nb := \max_{1\leq j\leq N} n_j$. Then, a $1-\alpha$ Hoeffding confidence interval for $\tau_n$ is given by
    \begin{equation*}
        \left[\tauhm_n \pm \frac{b-a}{2} \frac{N^3 \nb}{N_0N_1n}\sqrt{\frac{(1+\epsilon_N')\log(2/\alpha)}{2N}}\right].
    \end{equation*}
    Let $\vhn^*$ be Neyman's variance estimator with respect to the pseudo potential outcomes. An explicit $1-\alpha$ Bernstein confidence interval is given by
    \begin{equation*}
        \left[\tauhm_n \pm \left(\sqrt{\frac{2(1+\epsilon_N'')(1+\epsilon_N') \vhn^* \log(3/\alpha)}{N}} + \frac{C(N_0,N_1)\,\nb\,(b-a)}{n} \log(3/\alpha)\right)\right],
    \end{equation*}
    where $\epsilon_N''$ and $C(N_0, N_1)$ are defined analogously to Proposition~\ref{prop:cre-bernstein}.
\end{proposition}

\subsection{Stratified Randomized Experiments} The vanilla CRE ensures that the number of units assigned to treatment and control are fixed across the whole study population. Within strata (subgroups defined in terms of the covariates), the allocation may still be unbalanced which may reduce efficiency, particularly when the treatment effect is heterogeneous across strata. A popular approach to rule out lopsided treatment assignments, is the stratified completely randomized experiment (SCRE), which conducts independent CREs for each stratum.

The corresponding treatment assignment distribution for $J$ strata is denoted $Z \sim \mathrm{SCRE}(n_{1,0}, n_{1,1};\ldots;n_{J,0},n_{J,1})$ and defined as $Z:=(Z^*_{1}, \ldots,Z^*_{J})$, where $Z^*_{j}\sim \cre(n_{j,0},n_{j,1})$ for $j \in \{1,\ldots,J\}$ and the $Z^*_j$ are jointly independent. Accordingly, $n_j=n_{j,0}+n_{j,1}$ is the number of units in the $j$-th stratum, $n = \sum_{j=1}^J n_j$ is the total number of units and $s_j = n_j/n$ is the relative size of stratum $j$. The matched-pair design in Section~\ref{sec:matched-pairs} is a special case of this experiment because each pair is a stratum with two units. Constructing confidence intervals for a general SCRE proceeds differently, though, because we can estimate the variance within a stratum if at least two units are assigned to treatment and control, respectively.

We again adopt double indexing where $Z_{ji}$, and $Y_{ji}(0)$ and $Y_{ji}(1)$ denote the treatment and potential outcomes of the $i$-th unit in the $j$-the stratum, respectively, where $i\in\{1,\ldots,n_j\}$ and $j\in\{1,\ldots,J\}$. Since a completely randomized experiment is conducted within each cluster, all $m$-centred HT estimators are equal and we can express them as a weighted sum of the stratum-specific HT estimators:
\begin{align*}
    \tauh_n &= \frac{1}{n} \sum_{j=1}^J \sum_{i=1}^{n_j} Y_{ji}(1) \frac{Z_{ji}}{n_{j,1}/n_j} - Y_{ji}(0) \frac{1-Z_{ji}}{n_{j,0}/n_j}= \sum_{j=1}^J \frac{s_j}{n_j} \sum_{i=1}^{n_j} Y_{ji}(1) \frac{Z_{ji}^*}{n_{j,1}/n_j} - Y_{ji}(0) \frac{1-Z_{ji}^*}{n_{j,0}/n_j}\\
    &=: \sum_{j=1}^J s_j \, \tauh_{n_j}^j.
\end{align*}
Since the individual estimators $\tauh_{n_j}^j$ are independent, the variance of $\tauh_n$ is also a weighted sum of the stratum-level variances. Therefore, the natural extension of Neyman's variance estimator to SCREs is given by $\vhns := \sum_{j=1}^J s_j^2\, \vhn^j$, where $\vhn^j$ denotes the variance estimator in the $j$-th stratum. More details can be found in \citet{ding_first_2024}[chap.~5.3], for instance.

Since the CREs in the different clusters are independent, the moment generating function of $\tauh_n-\tau_n$ factorizes and we can ``aggregate'' the results of Sections~\ref{sec:cre-hoeffding} and~\ref{sec:bernstein} which apply to individual strata. For brevity, we only state the resulting Hoeffding and (explicit) Bernstein confidence intervals. The proof is deferred to Appendix~\ref{app:scre}.

\begin{proposition}\label{prop:scre}
    Suppose $Z \sim \mathrm{SCRE}(n_{1,0}, n_{1,1};\ldots; n_{J,0}, n_{J,1})$, where $n_{j,0},n_{j,1}\geq 2$ for all $j\in\{1,\ldots,J\}$, and let $\alpha \in (0,1)$. Let the $\epsilon'_{n_j}$ be defined as in Proposition~\ref{prop:cre-hoeffding}, and $\epsilon_{n_j}''$ as in Proposition~\ref{prop:cre-bernstein}. Then, a $1-\alpha$ Hoeffding confidence interval for $\tau_n$ is given by
    \begin{equation*}
        \left[\tauh_n \pm \frac{b-a}{2}\sqrt{\sum_{j=1}^J\frac{n_j^5}{n\,n_{j,0}^2n_{j,1}^2}(1+\epsilon_{n_j}')}\sqrt{\frac{\log(2/\alpha)}{2n}}\right].
    \end{equation*}
    Furthermore, an explicit $1-\alpha$ Bernstein confidence interval is given by
    \begin{equation*}
        \left[\tauh_n \pm \left(\sqrt{2 \max_j\left\{\frac{(1+\epsilon_{n_j}')(1+\epsilon_{n_j}'')}{n_j}\right\} \vhns \log(3/\alpha)} + \frac{C^\mathrm{S}(b-a)}{n}\log(3/\alpha)\right)\right],
    \end{equation*}
    where the constant is defined as
    \begin{multline*}
        C^\mathrm{S} := \frac{1}{3} \max_j \left\{\frac{(1+\epsilon_{n_j}')n(n_j-1)}{n_{j,0}\,n_{j,1}}\right\}\\
        + \sqrt{2\max_j\left\{\frac{(1+\epsilon_{n_j}')(1+\epsilon_{n_j}'')\,n}{n_j}\right\}} \,\,\max_{j} \left\{\frac{n_j}{n}\sqrt{c(n_{j,0},n_{j,1})}\right\}.
    \end{multline*}
\end{proposition}


\section{Discussion}\label{sec:discussion}
In this article, we investigated a non-asymptotic approach to constructing confidence intervals for the SATE under the most common experimental designs. Our key technical contributions are the improved concentration inequality for independent random variables, cf.\ Theorem~\ref{thm:mine-bernstein}, and the first concentration inequality for Neyman's variance estimator, cf.\ Theorem~\ref{thm:neyman-conc}. Since we aim to construct as tight intervals as possible, we made efforts to avoid using loose upper bounds on constants (including in terms of lower order) and to provide the implicit Bennett intervals along the explicit Bernstein intervals.

We compared the empirical performance of different Hoeffding and Bernstein-type confidence intervals and found substantial differences, in some cases. In practice, one can simply choose the shortest of the confidence intervals because all of them are centred at the HT estimator. As seen in Figure~\ref{fig:cre-bernstein}, in a completely randomized experiment, our non-asymptotic confidence intervals are considerably wider than the asymptotic Wald interval, even for moderately large samples sizes. In order to remedy this undesirable feature, more sophisticated concentration inequalities for dependent random variables are required which may be an interesting endeavour for probability theorists, as well.

This work could be extended in several directions, both in terms of estimation strategies and experimental designs. In our article, we only considered the HT estimator since it is linear in the treatment assignment which facilitates deriving concentration inequalities. Similar non-asymptotic confidence intervals based on \citet{hajek_1958,hajek_comment_1971}'s estimator would be desirable since it often has smaller variance, albeit being slightly biased. Furthermore, more recent work has used covariate information in the estimation to increase efficiency. \citet{freedman_regression_2008} and \citet{lin_agnostic_2013} use linear regression adjustment, and \citet{guo_generalized_2023} and \citet{cohen_no-harm_2024} extend this idea to more flexible models. In addition, \citet{miratrix_adjusting_2013} integrate covariate information in the estimation via post-stratification. We are optimistic that our techniques can be extended to these estimation strategies as the covariates are treated as fixed quantities under the design-based paradigm.

Beyond the common experimental designs we addressed in this article, deriving non-asymptotic confidence intervals under different treatment assignment schemes would be an interesting extension. These could include (fractional) factorial designs \citep{dasgupta_causal_2015, pashley_causal_2023}, stepped wedge experiments \citep{ji_randomization_2017} or re-randomized designs \citet{morgan_rerandomization_2012}. The latter, in particular, could pose technical difficulties as the randomization distribution is only known implicitly.

Another increasingly common class of designs are adaptive experiments. Initially introduced in \citet{thompson_likelihood_1933}'s landmark paper, they have become a popular approach in economics \citep{kasy_adaptive_2021}, political science \citep{offer-westort_adaptive_2021} and medicine \citep{burnett_adding_2020}. This group of designs includes for instance response-adaptive randomization, multi-arm multi-stage trials, and enrichment designs; for an overview article and a book-length treatment, we refer the reader to \citet{pallmann_adaptive_2018} and \citet{rosenberger_randomization_2015}, respectively. \citet{howard_time-uniform_2021} and~\citet{waudby-smith_estimating_2024} use martingale concentration inequalities in their respective works to derive non-asymptotic confidence sequences. These yield valid confidence intervals even when the experiment is stopped at an arbitrary point. This flexibility comes at the cost of interpretation as the estimand, i.e.\ the sample average treatment effect, also depends on the stopping time under the random design paradigm. Therefore, we think that more research on a good notion of reference population could increase the practical appeal of the design-based approach in adaptive experiments.

\begin{acks}[Acknowledgments]
I would like to thank Ricardo Sandoval, Ian Waudby-Smith and Myrto Limnios for very insightful discussions and feedback on earlier versions that greatly helped to shape this article.
\end{acks}

\begin{funding}
TF was supported by the Swiss National Science Foundation under project grant 207436.
\end{funding}

\bibliographystyle{imsart-nameyear} 
\bibliography{bib}       

@article{barber_hoeffding_2024,
	title = {Hoeffding and {Bernstein} inequalities for weighted sums of exchangeable random variables},
	volume = {29},
	journal = {Electronic Communications in Probability},
	author = {Barber, Rina Foygel},
	year = {2024},
	pages = {1--13}
}

@article{maurer_empirical_2009,
	title = {Empirical {Bernstein} {Bounds} and {Sample} {Variance} {Penalization}},
        journal = {arXiv: 0907.3740},
	author = {Maurer, Andreas and Pontil, Massimiliano},
	year = {2009}
}

@article{bardenet_concentration_2015,
	title = {Concentration inequalities for sampling without replacement},
	volume = {21},
	number = {3},
	journal = {Bernoulli},
	author = {Bardenet, Rémi and Maillard, Odalric-Ambrym},
	year = {2015},
	pages = {1361--1385}
}

@article{bertail_bernstein-type_2019,
	title = {Bernstein-type exponential inequalities in survey sampling: {Conditional} {Poisson} sampling schemes},
	volume = {25},
	number = {4B},
	journal = {Bernoulli},
	author = {Bertail, Patrice and Cl\'emen\c{c}on, Stephan},
	year = {2019},
	pages = {3527--3554}
}

@article{sambale_concentration_2022,
	title = {Concentration {Inequalities} on the {Multislice} and for {Sampling} {Without} {Replacement}},
	volume = {35},
	number = {4},
	journal = {Journal of Theoretical Probability},
	author = {Sambale, Holger and Sinulis, Arthur},
	year = {2022},
	pages = {2712--2737}
}

@book{boucheron_concentration_2013,
	title = {Concentration {Inequalities}: {A} {Nonasymptotic} {Theory} of {Independence}},
	publisher = {OUP Oxford},
	author = {Boucheron, Stéphane and Lugosi, Gábor and Massart, Pascal},
	year = {2013}
}

@article{adamczak_concentration_2021,
	title = {Concentration inequalities for some negatively dependent binary random variables},
	journal = {arXiv: 2108.12636},
	author = {Adamczak, Rados{\l}aw and Polaczyk, Bart{\l}omiej},
	year = {2021}
}

@inproceedings{waudby-smith_confidence_2020,
	title = {Confidence sequences for sampling without replacement},
	volume = {33},
	booktitle = {Advances in {Neural} {Information} {Processing} {Systems}},
	publisher = {Curran Associates, Inc.},
	author = {Waudby-Smith, Ian and Ramdas, Aaditya},
	year = {2020},
	pages = {20204--20214}
}

@article{splawa-neyman_application_1990,
	title = {On the {Application} of {Probability} {Theory} to {Agricultural} {Experiments}. {Essay} on {Principles}. {Section} 9},
	volume = {5},
	number = {4},
	journal = {Statistical Science},
	publisher = {Institute of Mathematical Statistics},
	author = {Sp{\l}awa-Neyman, Jerzy},
	year = {1923},
    note = {[translated from Polish by D.M. D{\k a}browska and T.P. Speed in 1990]},
	pages = {465--472}
}

@article{li_general_2017,
	title = {General {Forms} of {Finite} {Population} {Central} {Limit} {Theorems} with {Applications} to {Causal} {Inference}},
	volume = {112},
	number = {520},
	journal = {Journal of the American Statistical Association},
	author = {Li, Xinran and Ding, Peng},
	year = {2017},
	pages = {1759--1769}
}

@article{lee_logarithmic_1998,
	title = {Logarithmic {Sobolev} inequality for some models of random walks},
	volume = {26},
	number = {4},
	journal = {The Annals of Probability},
	publisher = {Institute of Mathematical Statistics},
	author = {Lee, Tzong-Yow and Yau, Horng-Tzer},
	year = {1998},
	pages = {1855--1873}
}

@article{salez_sharp_2021,
	title = {A sharp log-{Sobolev} inequality for the multislice},
	volume = {4},
	journal = {Annales Henri Lebesgue},
	author = {Salez, Justin},
	year = {2021},
	pages = {1143--1161}
}

@article{sandoval_nonasymptotic_2026,
	title = {On {Nonasymptotic} {Confidence} {Intervals} for {Treatment} {Effects} in {Randomized} {Experiments}},
    journal = {arXiv: 2601.11744},
	author = {Sandoval, Ricardo J. and Balakrishnan, Sivaraman and Feller, Avi and Jordan, Michael I. and Waudby-Smith, Ian},
	year = {2026}
}

@article{li_exact_2016,
	title = {Exact confidence intervals for the average causal effect on a binary outcome},
	volume = {35},
	number = {6},
	journal = {Statistics in Medicine},
	author = {Li, Xinran and Ding, Peng},
	year = {2016},
	pages = {957--960}
}

@article{rigdon_randomization_2015,
	title = {Randomization inference for treatment effects on a binary outcome},
	volume = {34},
	number = {6},
	journal = {Statistics in Medicine},
	author = {Rigdon, Joseph and Hudgens, Michael G.},
	year = {2015},
	pages = {924--935}
}

@article{fan_exponential_2015,
	title = {Exponential inequalities for martingales with applications},
	volume = {20},
	number = {},
	journal = {Electronic Journal of Probability},
	publisher = {Institute of Mathematical Statistics and Bernoulli Society},
	author = {Fan, Xiequan and Grama, Ion and Liu, Quansheng},
	year = {2015},
	pages = {1--22}
}

@article{howard_time-uniform_2021,
	title = {Time-{Uniform}, {Nonparametric}, {Nonasymptotic} {Confidence} {Sequences}},
	volume = {49},
	number = {2},
	journal = {The Annals of Statistics},
	publisher = {Institute of Mathematical Statistics},
	author = {Howard, Steven R. and Ramdas, Aaditya and McAuliffe, Jon and Sekhon, Jasjeet},
	year = {2021},
	pages = {1055--1080}
}

@article{hoeffding_probability_1963,
	title = {Probability {Inequalities} for {Sums} of {Bounded} {Random} {Variables}},
	volume = {58},
	number = {301},
	journal = {Journal of the American Statistical Association},
	author = {Hoeffding, Wassily},
	year = {1963},
	pages = {13--30}
}

@article{horvitz_generalization_1952,
	title = {A {Generalization} of {Sampling} {Without} {Replacement} from a {Finite} {Universe}},
	volume = {47},
	number = {260},
	journal = {Journal of the American Statistical Association},
	author = {Horvitz, D. G. and Thompson, D. J.},
	year = {1952},
	pages = {663--685}
}

@article{miratrix_worth_2018,
	title = {Worth {Weighting}? {How} to {Think} {About} and {Use} {Weights} in {Survey} {Experiments}},
	volume = {26},
	number = {3},
	journal = {Political Analysis},
	author = {Miratrix, Luke W. and Sekhon, Jasjeet S. and Theodoridis, Alexander G. and Campos, Luis F.},
	year = {2018},
	pages = {275--291}
}

@article{abadie_samplingbased_2020,
	title = {Sampling‐{Based} versus {Design}‐{Based} {Uncertainty} in {Regression} {Analysis}},
	volume = {88},
	number = {1},
	journal = {Econometrica},
	author = {Abadie, Alberto and Athey, Susan and Imbens, Guido W. and Wooldridge, Jeffrey M.},
	year = {2020},
	pages = {265--296}
}

@article{manski_how_2018,
	title = {How {Do} {Right}-to-{Carry} {Laws} {Affect} {Crime} {Rates}? {Coping} with {Ambiguity} {Using} {Bounded}-{Variation} {Assumptions}},
	volume = {100},
	number = {2},
	journal = {Review of Economics and Statistics},
	author = {Manski, Charles F. and Pepper, John V.},
	year = {2018},
	pages = {232--244}
}

@article{pitman_significance_1937,
	title = {Significance {Tests} {Which} {May} be {Applied} to {Samples} from {Any} {Populations}},
	volume = {4},
	number = {1},
	journal = {Journal of the Royal Statistical Society Series B: Statistical Methodology},
	author = {Pitman, E. J. G.},
	year = {1937},
	pages = {119--130}
}

@article{neyman_two_1934,
	title = {On the {Two} {Different} {Aspects} of the {Representative} {Method}: {The} {Method} of {Stratified} {Sampling} and the {Method} of {Purposive} {Selection}},
	volume = {97},
	number = {4},
	journal = {Journal of the Royal Statistical Society},
	author = {Neyman, Jerzy},
	year = {1934},
	pages = {558}
}

@book{cochran_sampling_1977,
	series = {Wiley {Ser}. {Probab}. {Math}. {Stat}.},
	title = {Sampling techniques},
    edition = {3rd},
	publisher = {John Wiley \& Sons, Hoboken, NJ},
	author = {Cochran, William G.},
	year = {1977}
}

@book{ding_first_2024,
	address = {New York},
	title = {A {First} {Course} in {Causal} {Inference}},
	publisher = {Chapman and Hall/CRC},
	author = {Ding, Peng},
	year = {2024}
}

@article{zhang_2023_randomization_test,
	title = {What is a {Randomization} {Test}?},
	volume = {0},
	number = {0},
	journal = {Journal of the American Statistical Association},
	author = {Zhang, Yao and Zhao, Qingyuan},
	year = {2023},
	pages = {1--15},
}

@book{imbens_causal_2015,
	address = {Cambridge},
	title = {Causal {Inference} for {Statistics}, {Social}, and {Biomedical} {Sciences}: {An} {Introduction}},
	publisher = {Cambridge University Press},
	author = {Imbens, Guido W. and Rubin, Donald B.},
	year = {2015}
}

@article{hajek_1960,
 author = {H\'ajek, Jaroslav},
 title = {Limiting distributions in simple random sampling from a finite population},
 journal = {Publications of the Mathematical Institute of the Hungarian Academy of Sciences, Series A},
 volume = {5},
 pages = {361--374},
 year = {1960}
}

@article{ding_bridging_2017,
	title = {Bridging {Finite} and {Super} {Population} {Causal} {Inference}},
	volume = {5},
	number = {2},
	journal = {Journal of Causal Inference},
	publisher = {De Gruyter},
	author = {Ding, Peng and Li, Xinran and Miratrix, Luke W.},
	year = {2017}
}

@article{imai_misunderstandings_2008,
	title = {Misunderstandings {Between} {Experimentalists} and {Observationalists} about {Causal} {Inference}},
	volume = {171},
	number = {2},
	journal = {Journal of the Royal Statistical Society Series A: Statistics in Society},
	author = {Imai, Kosuke and King, Gary and Stuart, Elizabeth A.},
	year = {2008},
	pages = {481--502}
}

@article{shi_berryesseen_2026,
	title = {Berry–{Esseen} bounds for design-based causal inference with possibly diverging treatment levels and varying group sizes},
	volume = {54},
	number = {1},
	journal = {The Annals of Statistics},
	publisher = {Institute of Mathematical Statistics},
	author = {Shi, Lei and Ding, Peng},
	year = {2026},
	pages = {324--349}
}

@article{yang_rejective_2023,
	title = {Rejective {Sampling}, {Rerandomization}, and {Regression} {Adjustment} in {Survey} {Experiments}},
	volume = {118},
	number = {542},
	journal = {Journal of the American Statistical Association},
	author = {Yang, Zihao and Qu, Tianyi and Li, Xinran},
	year = {2023},
	pages = {1207--1221}
}

@book{sarndal_1992,
 author = {S{\"a}rndal, Carl-Erik and Swensson, Bengt and Wretman, Jan},
 title = {Model assisted survey sampling},
 series = {Springer Series in Statistics},
 year = {1992},
 publisher = {New York etc.: Springer-Verlag}
}

@article{ding_what_2025,
	title = {What randomization can and cannot guarantee},
	volume = {11},
	number = {1},
	journal = {Observational studies},
	author = {Ding, Peng},
    year = {2025},
	pages = {27--40}
}

@article{aronow_nonparametric_2025,
	title = {Nonparametric identification is not enough, but randomized controlled trials are},
    year = {2025},
	volume = {11},
	number = {1},
	journal = {Observational studies},
	author = {Aronow, P. M. and Robins, James M. and Saarinen, Theo and S\"avje, Fredrik and Sekhon, Jasjeet S.},
	pages = {3--16},
}

@book{lehmann_2006,
 author = {Lehmann, Erich L.},
 title = {Nonparametrics: {Statistical} methods based on ranks},
 edition = {2nd},
 year = {2006},
 publisher = {New York, NY: Springer}
}

@book{fisher_design_1935,
	address = {Edinburgh},
	title = {The design of experiments},
	publisher = {Oliver \& Boyd},
	author = {Fisher, R. A.},
	year = {1935}
}

@article{ernst_permutation_2004,
	title = {Permutation {Methods}: {A} {Basis} for {Exact} {Inference}},
	volume = {19},
	number = {4},
	journal = {Statistical Science},
	author = {Ernst, Michael D.},
	year = {2004},
	pages = {676--685}
}

@article{kempthorne_behaviour_1969,
	title = {The behaviour of some significance tests under experimental randomization},
	volume = {56},
	number = {2},
	journal = {Biometrika},
	author = {Kempthorne, Oscar and Doerfler, T. E.},
	year = {1969},
	pages = {231--248}
}

@article{mercer_theory_2017,
	title = {Theory and {Practice} in {Nonprobability} {Surveys}: {Parallels} between {Causal} {Inference} and {Survey} {Inference}},
	volume = {81},
	number = {S1},
	journal = {Public Opinion Quarterly},
	author = {Mercer, Andrew W. and Kreuter, Frauke and Keeter, Scott and Stuart, Elizabeth A.},
	year = {2017},
	pages = {250--271}
}

@article{srndal_design-based_1978,
	title = {Design-{Based} and {Model}-{Based} {Inference} in {Survey} {Sampling} [with {Discussion} and {Reply}]},
	volume = {5},
	number = {1},
	journal = {Scandinavian Journal of Statistics},
	author = {S\"arndal, Carl-Erik and Thomsen, Ib and Hoem, Jan M. and Lindley, D. V. and Barndorff-Nielsen, O. and Dalenius, Tore},
	year = {1978},
	pages = {27--52}
}

@article{serfling_probability_1974,
	title = {Probability {Inequalities} for the {Sum} in {Sampling} without {Replacement}},
	volume = {2},
	number = {1},
	journal = {The Annals of Statistics},
	author = {Serfling, R. J.},
	year = {1974},
	pages = {39--48},
}

@article{waudby-smith_estimating_2024,
	title = {Estimating means of bounded random variables by betting},
	volume = {86},
	number = {1},
	journal = {Journal of the Royal Statistical Society Series B: Statistical Methodology},
	author = {Waudby-Smith, Ian and Ramdas, Aaditya},
	year = {2024},
	pages = {1--27}
}

@article{ramdas_game-theoretic_2023,
	title = {Game-{Theoretic} {Statistics} and {Safe} {Anytime}-{Valid} {Inference}},
	volume = {38},
	number = {4},
	journal = {Statistical Science},
	author = {Ramdas, Aaditya and Grünwald, Peter and Vovk, Vladimir and Shafer, Glenn},
	year = {2023}
}

@article{hajek_1958,
 author = {H{\'a}jek, Jaroslav},
 title = {On the theory of ratio estimates},
 journal = {Applications of Mathematics},
 volume = {3},
 pages = {384--398},
 year = {1958}
}

@book{hajek_comment_1971,
    title = {Comment on ‘Foundations of Statistical Inference: An Essay on the Logical Foundations of Survey Sampling, Part I’ by D. Basu},
    publisher = {Holt, Rinehart and Winston, Toronto, Canada},
    author = {H{\'a}jek, Jaroslav},
	year = {1971},
	pages = {236},
}

@article{rubin_estimating_1974,
	title = {Estimating causal effects of treatments in randomized and nonrandomized studies},
	volume = {66},
	number = {5},
	journal = {Journal of Educational Psychology},
	publisher = {American Psychological Association},
	author = {Rubin, Donald B.},
	year = {1974},
	pages = {688--701}
}

@book{hernan2020causal,
  author    = {Hernán, Miguel A. and Robins, James M.},
  title     = {Causal Inference: What If},
  publisher = {Chapman \& Hall/CRC},
  year      = {2020},
  address   = {Boca Raton}
}

@article{bernstein1924,
  author  = {Bernstein, S. N.},
  title   = {On a modification of {Chebyshev}'s inequality
             and on the error in {Laplace}'s formula},
  journal = {Uchen. Zapiski Nauchn.-issled. Kafedr Ukrainy, Otdel. Mat., vyp. 1},
  number  = {4},
  pages   = {38--49},
  year    = {1924},
  note    = {In Russian}
}

@book{bernstein1946,
  author    = {Bernstein, S. N.},
  title     = {Theory of Probability},
  edition   = {4th enlarged},
  publisher = {Gostekhizdat},
  address   = {Moscow--Leningrad},
  year      = {1946},
  note      = {In Russian}
}

@book{de_la_pena_self-normalized_2009,
	address = {Berlin, Heidelberg},
	series = {Probability and its {Applications}},
	title = {Self-{Normalized} {Processes}},
	publisher = {Springer},
	author = {de la Peña, Victor H. and Lai, Tze Leung and Shao, Qi-Man},
	editor = {Gani, Joe and Heyde, C. C. and Jagers, P. and Kurtz, T. G.},
	year = {2009}
}

@article{aronow_class_2013,
	title = {A {Class} of {Unbiased} {Estimators} of the {Average} {Treatment} {Effect} in {Randomized} {Experiments}},
	volume = {1},
	number = {1},
	journal = {Journal of Causal Inference},
	author = {Aronow, Peter M. and Middleton, Joel A.},
	year = {2013},
	pages = {135--154}
}

@article{aronow_minimax_2026,
	title = {Minimax unbiased estimation for finite populations with bounded outcomes},
    journal = {arXiv: 2605.20572},
	author = {Aronow, P. M. and Lopatto, Patrick},
	year = {2026}
}

@article{bennett_1962,
author = {George Bennett},
title = {Probability Inequalities for the Sum of Independent Random Variables},
journal = {Journal of the American Statistical Association},
volume = {57},
number = {297},
pages = {33--45},
year = {1962}
}

@article{greene_exponential_2017,
	title = {Exponential bounds for the hypergeometric distribution},
	volume = {23},
	number = {3},
	journal = {Bernoulli},
	author = {Greene, Evan and Wellner, Jon A.},
	year = {2017},
    pages = {1911--1950}
}

@article{su_model-assisted_2021,
	title = {Model-{Assisted} {Analyses} of {Cluster}-{Randomized} {Experiments}},
	volume = {83},
	number = {5},
	journal = {Journal of the Royal Statistical Society Series B: Statistical Methodology},
	author = {Su, Fangzhou and Ding, Peng},
	year = {2021},
	pages = {994--1015}
}

@article{schochet_design-based_2022,
	title = {Design-{Based} {Ratio} {Estimators} and {Central} {Limit} {Theorems} for {Clustered}, {Blocked} {RCTs}},
	volume = {117},
	number = {540},
	journal = {Journal of the American Statistical Association},
	author = {Schochet, Peter Z. and Pashley, Nicole E. and Miratrix, Luke W. and Kautz, Tim},
	year = {2022},
	pages = {2135--2146}
}

@article{middleton_unbiased_2015,
	title = {Unbiased {Estimation} of the {Average} {Treatment} {Effect} in {Cluster}-{Randomized} {Experiments}},
	volume = {6},
	number = {1-2},
	journal = {Statistics, Politics and Policy},
	author = {Middleton, Joel A. and Aronow, Peter M.},
	year = {2015}
}

@article{freedman_regression_2008,
	title = {On regression adjustments to experimental data},
	volume = {40},
	number = {2},
	journal = {Advances in Applied Mathematics},
	author = {Freedman, David A.},
	year = {2008},
	pages = {180--193}
}

@article{miratrix_adjusting_2013,
	title = {Adjusting {Treatment} {Effect} {Estimates} by {Post}-{Stratification} in {Randomized} {Experiments}},
	volume = {75},
	number = {2},
	journal = {Journal of the Royal Statistical Society Series B: Statistical Methodology},
	author = {Miratrix, Luke W. and Sekhon, Jasjeet S. and Yu, Bin},
	year = {2013},
	pages = {369--396}
}

@article{lin_agnostic_2013,
	title = {Agnostic notes on regression adjustments to experimental data: {Reexamining} {Freedman}’s critique},
	volume = {7},
	number = {1},
	journal = {The Annals of Applied Statistics},
	publisher = {Institute of Mathematical Statistics},
	author = {Lin, Winston},
	year = {2013},
	pages = {295--318}
}

@article{morgan_rerandomization_2012,
	title = {Rerandomization to improve covariate balance in experiments},
	volume = {40},
	number = {2},
	journal = {The Annals of Statistics},
	author = {Morgan, Kari Lock and Rubin, Donald B.},
	year = {2012},
	pages = {1263--1282}
}

@article{cohen_no-harm_2024,
	title = {No-harm calibration for generalized {Oaxaca}–{Blinder} estimators},
	volume = {111},
	issn = {1464-3510},
	number = {1},
	journal = {Biometrika},
	author = {Cohen, P L and Fogarty, C B},
	year = {2024},
	pages = {331--338},
}

@article{guo_generalized_2023,
	title = {The {Generalized} {Oaxaca}-{Blinder} {Estimator}},
	volume = {118},
    number = {541},
	journal = {Journal of the American Statistical Association},
	publisher = {Taylor \& Francis},
	author = {Guo, Kevin and Basse, Guillaume},
	year = {2023},
	pages = {524--536},
}

@article{dasgupta_causal_2015,
	title = {Causal {Inference} from {2K} {Factorial} {Designs} by {Using} {Potential} {Outcomes}},
	volume = {77},
	number = {4},
	journal = {Journal of the Royal Statistical Society Series B: Statistical Methodology},
	author = {Dasgupta, Tirthankar and Pillai, Natesh S. and Rubin, Donald B.},
	year = {2015},
	pages = {727--753}
}

@article{pashley_causal_2023,
	title = {Causal inference for multiple treatments using fractional factorial designs},
	volume = {51},
	number = {2},
	journal = {Canadian Journal of Statistics},
	author = {Pashley, Nicole E. and Bind, Marie‐Abèle C.},
	year = {2023},
	pages = {444--468},
}

@article{ji_randomization_2017,
	title = {Randomization inference for stepped-wedge cluster-randomized trials: {An} application to community-based health insurance},
	volume = {11},
	issn = {1932-6157, 1941-7330},
	number = {1},
	journal = {The Annals of Applied Statistics},
	author = {Ji, Xinyao and Fink, Gunther and Robyn, Paul Jacob and Small, Dylan S.},
	year = {2017},
	pages = {1--20}
}

@article{kasy_adaptive_2021,
	title = {Adaptive {Treatment} {Assignment} in {Experiments} for {Policy} {Choice}},
	volume = {89},
	number = {1},
	journal = {Econometrica},
	author = {Kasy, Maximilian and Sautmann, Anja},
	year = {2021},
	pages = {113--132}
}

@article{offer-westort_adaptive_2021,
	title = {Adaptive {Experimental} {Design}: {Prospects} and {Applications} in {Political} {Science}},
	volume = {65},
	shorttitle = {Adaptive {Experimental} {Design}},
	number = {4},
	journal = {American Journal of Political Science},
	author = {Offer-Westort, Molly and Coppock, Alexander and Green, Donald P.},
	year = {2021},
	pages = {826--844}
}

@article{thompson_likelihood_1933,
	title = {On the likelihood that one unknown probability exceeds another in view of the evidence of two samples},
	volume = {25},
	number = {3-4},
	journal = {Biometrika},
	author = {Thompson, William R.},
	year = {1933},
	pages = {285--294}
}

@book{rosenberger_randomization_2015,
	title = {Randomization in {Clinical} {Trials}: {Theory} and {Practice}},
	shorttitle = {Randomization in {Clinical} {Trials}},
	publisher = {John Wiley \& Sons},
	author = {Rosenberger, William F. and Lachin, John M.},
	year = {2015}
}

@article{pallmann_adaptive_2018,
	title = {Adaptive designs in clinical trials: why use them, and how to run and report them},
	volume = {16},
	shorttitle = {Adaptive designs in clinical trials},
	number = {1},
	journal = {BMC Medicine},
	author = {Pallmann, Philip and Bedding, Alun W. and Choodari-Oskooei, Babak and Dimairo, Munyaradzi and Flight, Laura and Hampson, Lisa V. and Holmes, Jane and Mander, Adrian P. and Odondi, Lang’o and Sydes, Matthew R. and Villar, Sofía S. and Wason, James M. S. and Weir, Christopher J. and Wheeler, Graham M. and Yap, Christina and Jaki, Thomas},
	year = {2018},
	pages = {29}
}

@article{burnett_adding_2020,
	title = {Adding flexibility to clinical trial designs: an example-based guide to the practical use of adaptive designs},
	volume = {18},
    shorttitle = {Adding flexibility to clinical trial designs},
	number = {1},
	journal = {BMC Medicine},
	author = {Burnett, Thomas and Mozgunov, Pavel and Pallmann, Philip and Villar, Sofia S. and Wheeler, Graham M. and Jaki, Thomas},
	year = {2020},
	pages = {352},
}


\newpage
\begin{appendix}

\section{Extensions of Concentration Inequalities in the Literature}\label{app:extension-literature}

This section provides slight generalizations of concentration inequalities of~\cite{barber_hoeffding_2024} and~\cite{bardenet_concentration_2015}. This mostly concerns using an arbitrary support $[l,u]$, sharpening constants if possible and, for Bernstein inequalities, also proving the corresponding, tighter Bennett-type bounds.

We start with an extension of \cite{barber_hoeffding_2024}'s Hoeffding inequality stated in their theorem~3.1.
\begin{theorem}\label{prop:rina-hoeffding}
    Let $w \in \R^n$ and $X_1,\ldots,X_N \in [l,u]$ be exchangeable random variables, where $N \geq n \geq 2$. Define $\bar{X} := \frac{1}{N}\sum_{i=1}^N X_i$ and $\epsilon'_N := \frac{\epsilon_N-1}{N-\epsilon_N} = \mathcal{O}(\frac{\log(N)}{N})$, where $\epsilon_N = \sum_{j=1}^N \frac{1}{j}$. Then, for any $\lambda \in \R$,
    \begin{equation*}
        \E\left[\exp\left(\lambda \sum_{i=1}^n w_i (X_i - \bar{X}) \right)\right] \leq \exp\left(\frac{\lambda^2(u-l)^2}{8}\,\lVert w\rVert_2^2\,(1+\epsilon_N')\right).
    \end{equation*}
    Moreover, for any $\delta \in (0,1)$, with probability $1-\delta$,
    \begin{equation*}
        \left\vert\sum_{i=1}^n w_i (X_i - \bar{X})\right\vert  \leq \lVert w \rVert_2 (u-l) \sqrt{\frac{(1+\epsilon'_N)\log(2/\delta)}{2}}.
    \end{equation*}
\end{theorem}

\begin{proof}
    The Hoeffding bound is based on Proposition~A.1 in~\citet{barber_hoeffding_2024}. Tracing the more general, original proof of this result in~\citet{waudby-smith_confidence_2020}, we find that
    \begin{equation*}
        M_t = \prod_{i=1}^t \exp\left(\lambda v_i (X_i-\bar{X}_{\geq i})-\frac{\lambda^2 (u-l)^2 v_i^2 }{8}\right)
    \end{equation*}
    is a super-martingale for any $\lambda \in \R$. Now, the proposition immediately follows from replacing the factor $\frac{\lambda^2}{2}$ in the proof of~\citet{barber_hoeffding_2024}[Thm.~3.1] with $\frac{\lambda^2(u-l)^2}{8}$.
\end{proof}

Next, we generalize \cite{barber_hoeffding_2024}'s Bernstein inequality given in their theorem~3.3
\begin{theorem}\label{thm:rina-bernstein}
    Let $w \in \R^n$ and $X_1,\ldots,X_n \in [l,u]$ be exchangeable random variables, where $N \geq n \geq 2$. Define $\bar{X}$ and $\epsilon_N'$ as in Theorem~\ref{prop:rina-hoeffding} and set $\sigma_X^2:= \frac{1}{N}\sum_{i=1}^N(X_i-\bar{X})^2$. Then, for any $\lambda \in \R$ such that $\lvert \lambda \rvert < 3/((u-l)\lVert w \rVert_\infty (1+\epsilon_N'))$,
    \begin{equation*}
        \E\left[\exp\left(\lambda \sum_{i=1}^n w_i (X_i - \bar{X}) - \frac{\lambda^2(1+\epsilon_N')\,\sigt^2_{X,N} \lVert w\rVert_2^2}{2\big(1-\frac{(u-l)\lvert \lambda \rvert }{3} \lVert w\rVert_\infty (1+\epsilon_N')\big)} \right)\right] \leq 1,
    \end{equation*}
     where $\sigt_{X,N}^2 = \sigma^2_X+4\epsilon_N'$. Furthermore, for $\delta \in (0,1)$, with probability $1-\delta$,
    \begin{equation*}
        \left\vert\sum_{i=1}^n w_i (X_i - \bar{X})\right\vert \leq \sigt_{X,N} \lVert w\rVert_2 \sqrt{2(1+\epsilon'_N)\log(2/\delta)}\, + \frac{u-l}{3}\lVert w\rVert_\infty(1+\epsilon_N')\log(2/\delta).
    \end{equation*}
    Define $\phi(u) := e^u -u-1$. For any $\lambda >0$,
    \begin{equation*}
        \E\left[\exp\left(\lambda \sum_{i=1}^n w_i (X_i - \bar{X}) -  \frac{\phi\big((1+\epsilon'_N)\lambda(u-l)\lVert w \rVert_\infty\big)}{(u-l)^2\lVert w\rVert_\infty^2} \frac{\sigt^2_{X,N} \lVert w \rVert_2^2}{1+\epsilon_N'}\right)\right] \leq 1.
    \end{equation*}
    Lastly, for any $\delta \in (0,1)$, with probability $1-\delta$,
    \begin{equation*}
        \left\vert\sum_{i=1}^n w_i (X_i - \bar{X})\right\vert \leq \frac{\sigt^2_{X,N} \lVert w \rVert_2^2}{(u-l)\lVert w \rVert_\infty}\, h_2^{-1}\left(\frac{(1+\epsilon_N')(u-l)^2\lVert w\rVert^2_\infty\log(2/\delta)}{\sigt^2_{X,N}\lVert w \rVert_2^2}\right),
    \end{equation*}
    where $h_2(u) := (1+u)\log(1+u)-u$.
\end{theorem}

\begin{proof}
    The first, explicit concentration inequality follows from a minor modification of the original proof. Since $X_i \in [l,u]$ for all $i \in \{1,\ldots,N\}$, we have $\lvert v_i (X_i-\bar{X}_{\geq i})\rvert \leq (u-l) \lVert v \rVert_\infty$ for any $v \in \R^N$. Now, the result directly follows from replacing the factor $\frac{2}{3}$ with $\frac{u-l}{3}$ in the proofs of \cite{barber_hoeffding_2024}[Prop.~B.1., Thm.~3.3].

    For the second, implicit concentration inequality we use \citet{bennett_1962}'s inequality in place of a Bernstein bound in the proofs of \cite{barber_hoeffding_2024}. First, we consider their Proposition~B.1. Replacing the Bernstein with a Bennett inequality \citep{boucheron_concentration_2013}[Thm.~2.9], we can show that the process 
    \begin{equation*}
        M_t = \prod_{i=1}^t \exp\left(\lambda v_i(X_i - \bar{X}_{\geq i}) - \frac{v_i^2\sigma^2_{X_{\geq i}}}{(u-l)^2\lVert v\rVert_\infty^2}\phi\Big(\lambda (u-l)\lVert v\rVert_\infty\Big)\right)
    \end{equation*}
    is a super-martingale for all $\lambda > 0$. Next, we follow the same steps as in the proof of \citeauthor{barber_hoeffding_2024}'s Theorem~3.3 with the factor $\phi(\lambda (u-l)\lVert v\rVert_\infty)/((u-l)^2\lVert v\rVert_\infty^2)$ instead of $\lambda^2/(2(1-\frac{2\lvert\lambda\rvert}{3}\lVert v \rVert_\infty))$. Replacing $\lambda$ with $(1+\epsilon_N')\lambda$ as in the original proof, we obtain the statement
    \begin{equation*}
        \E\left[\exp\left(\lambda \sum_{i=1}^n w_i (X_i-\bar{X})\right) \bigg\vert\,\, \sigt^2_{X,N}\right] \leq \exp\left( \frac{\phi\big((1+\epsilon'_N)\lambda(u-l)\lVert w \rVert_\infty\big)}{(u-l)^2\lVert w\rVert_\infty^2} \frac{\sigt^2_{X,N} \lVert w \rVert_2^2}{1+\epsilon_N'}\right).
    \end{equation*}
    Using the tail bound of \citet{boucheron_concentration_2013}[Thm.~2.9] with $b={(1+\epsilon'_N)}(u-l)\lVert w \rVert_\infty$ and  $v= (1+\epsilon'_N)\,\sigt^2_{X,N} \lVert w \rVert_2^2$, we see that for all $t>0$,
    \begin{equation*}
        \PR\left(\sum_{i=1}^n w_i (X_i-\bar{X}) \geq t\,\, \bigg \vert\, \sigt_{X,N}^2\right) \leq \exp\left(-\frac{\sigt^2_{X,N} \lVert w \rVert_2^2}{(1+\epsilon'_N)(u-l)^2\lVert w \rVert_\infty^2} h_2\left(\frac{(u-l)\lVert w \rVert_\infty \,t}{\sigt^2_{X,N} \lVert w \rVert_2^2}\right)\right).
    \end{equation*}
    We can now set the right-hand side equal to $\delta/2$, solve for $t$ and marginalize over $\sigt_{X,N}^2$. Doing the same with the signs of the weights $w_i$ reversed and combining the two inequalities with a union bound, we get
    \begin{equation*}
        \PR\left(\bigg\vert\sum_{i=1}^n w_i (X_i-\bar{X})\bigg\vert \geq \frac{\sigt^2_{X,N}\lVert w \rVert_2^2}{(u-l)\lVert w \rVert_\infty}\,h_2^{-1}\left(\frac{(1+\epsilon_N')(u-l)^2\lVert w\rVert^2_\infty\log(2/\delta)}{\sigt^2_{X,N}\lVert w \rVert_2^2}\right)\right) \leq \delta.
    \end{equation*}
\end{proof}

Lastly, we address \cite{bardenet_concentration_2015}'s Bernstein inequality stated in their corollary~3.6 and restrict ourselves to the case where no more than half of the population is sampled.
\begin{theorem}\label{thm:bm}
    Let $\mathcal{X}:=(x_1,\ldots,x_n)$ be a finite collection of real-valued numbers. Assume that they have mean $0$, are contained in the interval $[-B,B]$, where $B>0$, and denote their variance $\sigma^2 = \frac{1}{n} \sum_{i=1}^n x_i^2$. Further, let $X_1,\ldots,X_k$ be a random sample from $\mathcal{X}$ without replacement, where $k\leq n/2$ and let $\delta \in (0,1)$. Then, with probability $1-3\delta$,
    \begin{equation*}
        \left\vert\frac{1}{k}\sum_{i=1}^k X_i\right\vert \leq \sigma \sqrt{\frac{n-k+1}{n\,k}2\log(1/\delta)}+\left(\frac{4}{3k} + \sqrt{\frac{k-1}{n\,k\,(n-k+1)}}\right) B\log(1/\delta).
    \end{equation*}
    Moreover, with probability $1-3\delta$,
    \begin{equation*}
        \left\vert\frac{1}{k}\sum_{i=1}^k X_i\right\vert \leq \frac{v}{2B}\, h^{-1}_2\!\!\left(\frac{4B^2}{v}\frac{\log(1/\delta)}{k}\right),
    \end{equation*}
    where $h_2(u):=(1+u)\log(1+u)-u$ and
    \begin{equation*}
        v = \frac{k^2}{(n-k)^2}\left[\frac{n-k+1}{n}\sigma^2+\frac{\sigma B(k-1)}{n}\sqrt{\frac{2\log(1/\delta)}{k-1}}\right].
    \end{equation*}
\end{theorem}

\begin{proof}
    We prove the two results separately.
    
    \textit{Bernstein Bound} The first statement follows from Corollary~3.6 in \cite{bardenet_concentration_2015} with a few minor modifications. Their $n$ and $N$ correspond to our $k$ and $n$; moreover, we only consider the case $k\leq n/2$. Lastly, our random variables are centred ($\mu=0$) which allows us to replace the factor $b-a$ with $B$; see also the proofs of their Proposition~3.2 and Lemma~3.3.

    \textit{Bennett Bound} This result can be proven using the derivations in the proof of \cite[Thm.~3.5]{bardenet_concentration_2015}. To help the reader understand the structure of the argument, we provide the main steps here and refer to \cite{bardenet_concentration_2015} for the details. Note that we adapt the notation to our setting and replace $b-a$ with $B$ as explained above. References in the following paragraph invoke results from \cite{bardenet_concentration_2015} and not this article.

    Define $\varphi(x) := (e^x-1-x)/x^2 $ and let $\lambda >0$. Combining Proposition~3.2 and equation~(17) in Lemma~3.3, we obtain that with probability $1-2\delta$
    \begin{align*}
        \max_{1\leq j \leq k}\, \frac{1}{j}\sum_{i=1}^j X_j &\leq \frac{\log(1/\delta)}{\lambda} + \frac{\lambda}{k^2}\varphi\left(\frac{2B\lambda}{k}\right) \left[\sigma^2+\frac{\sigma B(k-1)}{n-k+1}\sqrt{\frac{2\log(1/\delta)}{k-1}}\right]\sum_{j=1}^k \frac{k^2}{(n-j)^2}\\
        &\leq \frac{\log(1/\delta)}{\lambda} + \frac{\lambda}{k^2}\varphi\left(\frac{2B\lambda}{k}\right) (k+1) \gamma^2.
    \end{align*}
    The second inequality follows from the estimation in~(10) and we define
    \begin{equation*}
        \gamma^2 := \left[\sigma^2+\frac{\sigma B(k-1)}{n-k+1}\sqrt{\frac{2\log(1/\delta)}{k-1}}\right] \frac{1}{k+1} \frac{k^3}{(n-k)^2}\left(1-\frac{k-1}{n}\right).
    \end{equation*}
    We can invert this statement above and get
    \begin{equation*}
        \PR\left(\max_{1\leq j \leq k}\, \frac{1}{j}\sum_{i=1}^j X_j \geq t\right)\leq \exp\left(-\lambda t +\frac{\lambda^2(k+1)}{k^2}\varphi\left(\frac{2B\lambda}{k}\right) \gamma^2\right) +\delta.
    \end{equation*}
    for all $t>0$. We can now follow the same steps as in \citeauthor{bardenet_concentration_2015}'s article to find the optimal $\lambda$, where $\gamma^2$ plays the role of $\tilde{\gamma}^2$. Plugging this optimal choice of $\lambda$ into the inequality above yields
     \begin{equation*}
        \PR\left(\max_{1\leq j \leq k}\, \frac{1}{j}\sum_{i=1}^j X_j \geq t\right) \leq \exp\left(-\frac{k}{2Bu}h(ut)\right)+\delta,
    \end{equation*}
    where
    \begin{equation*}
        u := \frac{2B(n-k)^2/k^2}{\frac{n-k+1}{n}\sigma^2+\frac{\sigma B(k-1)}{n}\sqrt{\frac{2\log(1/\delta)}{k-1}}}.
    \end{equation*}
   Solving the right-hand side for $t$, we obtain
    \begin{equation*}
        t = \frac{1}{u}\, h^{-1}_2\!\!\left(\frac{2Bu\log(1/\delta)}{k}\right).
    \end{equation*}
    Hence, with probability $1-2\delta$,
    \begin{equation*}
        \frac{1}{k}\sum_{i=1}^k X_i \leq \frac{1}{u}\, h^{-1}_2\!\!\left(\frac{2Bu\log(1/\delta)}{k}\right)
    \end{equation*}
    Lastly, we can repeat the same derivations for the random variables $-X_j$ to get a lower bound and combine the two via a union bound. Note that we need to account for the upper bound on the variance only once. We state the final result in terms of $v := \frac{2B}{u}$.
\end{proof}

\section{Proofs of Section~\ref{sec:bernoulli}}
\subsection{Self-Normalizing Concentration Inequality}\label{app:self-normalizing-conc-inq}

\begin{proof}[Proof of Theorem~\ref{thm:mine-bernstein}] We initially prove the concentration inequality for the case $[l_i,u_i] = [0,1]$ for all $i \in \{1,\ldots,n\}$. At the end of the proof, we show how the result in this ``base case'' can be directly extended to the general setting by a linear transformation of the random variables.

    Define the function $\psi(\lambda) := -\log(1-\lambda)-\lambda$, let $\lambda_1,\ldots,\lambda_n \in (0,1)$ be deterministic numbers and define $\xb_{<i} := \frac{1}{i-1}\sum_{j=1}^{i-1} X_j$ for $i \geq 2$ and $\xb_{<1}:=0$. First, we show that the process
    \begin{equation*}
        M_t := \prod_{i=1}^t \exp \left(\lambda_i(X_i-\mu_i)-\psi(\lambda_i)(X_i-\xb_{<i})^2\right)
    \end{equation*}
    is a super-martingale with respect to the filtration $\mathcal{F}_t:=\sigma(X_1,\ldots,X_t)$. A very similar result was shown by \citet[Thm.~4]{howard_time-uniform_2021}: They set all $\lambda_i$ to be equal and use a general predictable sequence $\hat{X}_i$ where we use $\xb_{<i}$.

    It is easy to see that $(M_t)_t$ is adapted to the filtration and the integrability condition is fulfilled as the $X_i$ are bounded. Hence, it remains to show that $\E[M_{t}\mid \mathcal{F}_{t-1}] \leq M_{t-1}$ almost surely for all $t\in\{1,\ldots,n\}$. First, we use the product structure of $(M_t)_t$ and measurability to re-write the expression:
    \begin{align*}
        \E[M_{t}\mid \mathcal{F}_{t-1}] &= M_{t-1} \cdot \E\left[\exp\left(\lambda_{t}(X_{t}-\mu_{t}) - \psi(\lambda_{t})(X_{t}-\xb_{<t})^2\right)\mid \mathcal{F}_{t-1}\right]\\
        &= M_{t-1} \cdot e^{\lambda_t(\xb_{<t}-\mu_t)} \E\left[\exp\left(\lambda_{t}(X_{t}-\xb_{<t}) - \psi(\lambda_{t})(X_{t}-\xb_{<t})^2\right)\mid \mathcal{F}_{t-1}\right].
    \end{align*}
    In the proof of their Lemma~4.1, \citet[eq.~(4.12)]{fan_exponential_2015} show that for any $\lambda \in[0,1)$ and $\xi \geq -1$ the inequality $\exp(\lambda \xi -\psi(\lambda)\xi^2)\leq 1+\lambda \xi$ holds. Applying this result with $\xi = X_t -\xb_{<t}$ inside the expectation yields
    \begin{align*}
        \E[M_{t}\mid \mathcal{F}_{t-1}] &\leq M_{t-1} \cdot e^{\lambda_t(\xb_{<t}-\mu_t)} \E\left[1+\lambda (X_{t}-\xb_{<t})\mid \mathcal{F}_{t-1}\right]\\
        &= M_{t-1} \cdot e^{\lambda_t(\xb_{<t}-\mu_t)} (1-\lambda_t (\xb_{<t}-\mu_t))
        \leq M_{t-1} \cdot e^{\lambda_t(\xb_{<t}-\mu_t)} e^{-\lambda_t(\xb_{<t}-\mu_t)}\\
        &= M_{t-1},
    \end{align*}
    where the last inequality follows from the fact that $1-x \leq e^{-x}$.

    Since $(M_t)_t$ is a super-martingale and $M_0=1$, we get the inequality
    \begin{equation}\label{eq:mgf-asym}
        \E\left[\exp\left(\sum_{i=1}^n\lambda_i(X_i-\mu_i)-\psi(\lambda_i)\sum_{i=1}^n(X_i-\xb_{<i})^2\right)\right] \leq 1.
    \end{equation}
    We could already obtain concentration results from this bound; these, however, would depend on the order of the data points due to the $\xb_{<i}$-terms. To remedy this issue, we use a symmetrization-technique similar to \citet{barber_hoeffding_2024}'s.

    To this end, we introduce the following notation: $\lambda := (\lambda_1,\ldots,\lambda_n)^T$, $D:=\diag(\psi(\lambda_1),\ldots,\psi(\lambda_n))$ and the matrix $A$ whose first row is given by $(1,\zero_{n-1})$ and whose $i$-th row is given by $(-\frac{1}{i-1} \one_{i-1},1,\zero_{n-i})$, where $i \in \{2,\ldots,n\}$. Note that $(AX)_i = X_i - \xb_{<i}$. We can now rewrite inequality~\eqref{eq:mgf-asym} as follows
    \begin{equation*}
        \E\left[\exp\left(\lambda^T(X-\mu)- X^T A^T D A X\right)\right] \leq 1.
    \end{equation*}
    Since this inequality holds for any order of the random variables $X_i$, it remains true when we take an average over all permutations of the vector $X$; the permutation matrix is denoted $\Pi$. Moreover, we can apply Jensen's inequality to move the average inside the exponential function yielding
    \begin{equation}\label{eq:mgf-pre-sym}
        \E\left[\exp\left(\frac{1}{n!}\sum_{\Pi}\lambda^T\Pi(X-\mu)- X^T\Pi^T A^T D A \,\Pi\,X\right)\right] \leq 1.
    \end{equation}
    The first term in the average easily simplifies to
    \begin{equation*}
        \frac{1}{n!}\sum_{\Pi}\lambda^T\Pi(X-\mu) = \bar{\lambda} \sum_{i=1}^n X_i-\mu_i,
    \end{equation*}
    where $\bar{\lambda} = \frac{1}{n}\sum_{i=1}^n \lambda_i$. For the second term, we invoke Lemma~\ref{lem:perm-avg} with $B=A^T D\, A$: We compute the trace
    \begin{equation*}
        \tr(A^T D\,A)= \tr(D AA^T)= \psi(\lambda_1) + \sum_{i=2}^n \psi(\lambda_i) \left(\frac{i-1}{(i-1)^2}+1\right) = \psi(\lambda_1) + \sum_{i=2}^n \psi(\lambda_i) \frac{i}{i-1},
    \end{equation*}
    and the sum of all the entries
    \begin{align*}
        \one^T (A^TD\,A) \one = e_1^T D\, e_1 = \psi(\lambda_1), 
    \end{align*}
    where $e_1 = (1,0,\ldots,0)^T$, and plug these quantities into the formula provided
    \begin{align*}
        \frac{1}{n!} \sum_{\Pi} \Pi^T A^T D\, A\, \Pi &= \frac{\tr(A^T D\, A)}{n} I +\frac{\one^TA^T D\, A\,\one-\tr(A^T D\, A)}{n(n-1)}(\one\one^T-I)\\
        &= \frac{\psi(\lambda_1)+\sum_{i=2}^n \psi(\lambda_i) \frac{i}{i-1}}{n}I - \frac{\sum_{i=2}^n \psi(\lambda_i) \frac{i}{i-1}}{n(n-1)}(\one\one^T-I)\\
        &=\frac{\psi(\lambda_1)}{n}I + \frac{\sum_{i=2}^n \psi(\lambda_i) \frac{i}{i-1}}{n-1}\left(I-\frac{1}{n}\one\one^T\right).
    \end{align*}
    The matrix $P^\perp:= I-\frac{1}{n}\one\one^T$ is the projection onto the subspace which is perpendicular to $\one$ and, consequently, $(P^\perp X)_i = X_i-\xb$ for all $i \in \{1,\ldots,n\}$. For this reason, we only want to keep the second term in the equation and set $\lambda_1 = 0$ and $\lambda_i = \lambda$ for some $\lambda \in (0,1)$ and all $i\in\{2,\ldots,n\}$. Thus, we obtain
    \begin{equation*}
        \frac{1}{n!} \sum_{\Pi} \Pi^T A^T D\, A\, \Pi
        =  \psi(\lambda) \frac{\sum_{i=2}^n 1+ \frac{1}{i-1}}{n-1} P^\perp
        = \psi(\lambda) \left(1+\frac{H_{n-1}}{n-1}\right) P^\perp.
    \end{equation*}
    Inserting this expression into~\eqref{eq:mgf-pre-sym}, we get
    \begin{equation}
        \E\left[\exp\left(\frac{n-1}{n}\lambda \sum_{i=1}^n(X_i-\mu_i)-\psi(\lambda)(1+\epsilon_{n-1})\sum_{i=1}^n(X_i-\bar{X})^2\right)\right] \leq 1,
    \end{equation}
    for all $\lambda \in [0,1)$. We can now apply Markov's inequality which implies
    \begin{equation*}
        \PR\left(\frac{n-1}{n}\lambda \sum_{i=1}^n(X_i-\mu_i)-\psi(\lambda)(1+\epsilon_{n-1})\sum_{i=1}^n(X_i-\bar{X})^2 \geq t\right) \leq e^{-t},
    \end{equation*}
    for any $t > 0$. Re-arranging this inequality and solving the right-hand side for $\delta$, we obtain
    \begin{equation*}
        \PR\left(\frac{1}{n}\sum_{i=1}^n X_i - \mu_i \geq \frac{\psi(\lambda)(1+\epsilon_{n-1})\vh}{\lambda} + \frac{\log(1/\delta)}{\lambda(n-1)}\right) \leq \delta.
    \end{equation*}
    Moreover, we notice that we can repeat the entire derivation for the random variables $-X_i$ with means $-\mu_i$ and obtain the same concentration result. Combining the inequalities for both sides with a union bound, we get that with probability $1-\delta$,
    \begin{equation}\label{eq:bound-lambda}
        \left\vert\frac{1}{n}\sum_{i=1}^n X_i - \mu_i\right\vert \leq \frac{\psi(\lambda)(1+\epsilon_{n-1})\vh}{\lambda} + \frac{\log(2/\delta)}{\lambda(n-1)}.
    \end{equation}
    Lastly, we choose the hyper-parameter $\lambda$ so that the bound is as tight as possible. We pursue two approaches: one that produces an explicit but slightly loose bound and one that provides a tighter, however implicit bound.
    
    \textit{Explicit Bound} We use the inequality $\psi(\lambda) \leq \frac{\lambda^2}{2(1-\lambda)}$ which holds for all $\lambda \in [0,1)$. Applying it to the right-hand side of~\eqref{eq:bound-lambda}, we obtain
    \begin{equation*}
        \frac{\lambda(1+\epsilon_{n-1})\vh}{2(1-\lambda)} + \frac{\log(2/\delta)}{\lambda(n-1)}.
    \end{equation*}
    Taking the derivative, equating it to zero and solving for $\lambda$, we find that the minimum is assumed at
    \begin{equation*}
        \lambda^* = \left(\sqrt{\frac{(1+\epsilon_{n-1})\vh}{2} \frac{n-1}{\log(2/\delta)}}+1\right)^{-1}.
    \end{equation*}
    Plugging $\lambda^*$ into the expression above it, we get the first result that with probability $1-\delta$,
    \begin{equation*}
        \left\vert\frac{1}{n}\sum_{i=1}^n X_i - \mu_i\right\vert \leq\sqrt{\frac{2(1+\epsilon_{n-1})\,\vh \,\log(2/\delta)}{n-1}} + \frac{\log(2/\delta)}{n-1}.
    \end{equation*}

    \textit{Implicit Bound} We directly take the derivative of the right-hand side of~\eqref{eq:bound-lambda} and equate it to zero
    \begin{equation*}
        \frac{1}{\lambda^2}\left((1+\epsilon_{n-1})\vh\, (\lambda \psi'(\lambda)-\psi(\lambda))-\frac{\log(2/\delta)}{n-1}\right) = 0,
    \end{equation*}
    which is equivalent to the equation
    \begin{equation*}
        \frac{\log(2/\delta)}{n-1} =  (1+\epsilon_{n-1})\vh \left(\frac{\lambda}{1-\lambda}+\log(1-\lambda)\right).
    \end{equation*}
    We now re-parameterize the equation above setting $\eta = \frac{\lambda}{1-\lambda}$ (and $\lambda = \frac{\eta}{1+\eta}$). Hence, we obtain that the optimal $\eta$ is implicitly given by
    \begin{equation*}
        \eta - \log(1+\eta) = \frac{\log(2/\delta)}{(1+\epsilon_{n-1})(n-1)\vh}.
    \end{equation*}
    We can now use the re-parameterization as well as the implicit definition of $\eta$ and plug these into the right-hand side of~\eqref{eq:bound-lambda}:
    \begin{multline*}
        \frac{1+\eta}{\eta} \left[(1+\epsilon_{n-1})\vh\left(\log(1+\eta)-\frac{\eta}{1+\eta}\right)+\frac{\log(2/\delta)}{n-1}\right]\\
        \begin{aligned}[t]
            &= \frac{1+\eta}{\eta} \left[(1+\epsilon_{n-1})\vh\left(\eta-\frac{\log(2/\delta)}{(1+\epsilon_{n-1})(n-1)\vh}-\frac{\eta}{1+\eta}\right)+\frac{\log(2/\delta)}{n-1}\right]\\
            &= \frac{1+\eta}{\eta} \left[\eta - \frac{\eta}{1+\eta}\right] (1+\epsilon_{n-1})\vh = (1+\epsilon_{n-1})\vh\, \eta
        \end{aligned}
    \end{multline*}
    Thus, we obtain the second statement that with probability $1-\delta$,
    \begin{equation*}
        \left\vert\frac{1}{n}\sum_{i=1}^n X_i - \mu_i\right\vert \leq (1+\epsilon_{n-1})\vh\, \eta.
    \end{equation*}

    In the general setting, where $X_i \in [l_i,u_i]$, we can apply our bounds above to the transformed random variable $\frac{1}{2}+(X_i-c_i)/L \in [0,1]$ and the corresponding generalized versions directly follow.
\end{proof}

\subsection{General Bernstein-Type Concentration Inequality}\label{app:bernoulli-general-bernstein}
\begin{proposition}
    Suppose $Z\sim \bern(\pi_1,\ldots,\pi_n)$ and let $\alpha \in (0,1)$. Define $\piu := \min_{1 \leq i \leq n} \{\pi_i, 1-\pi_i\}$,
    \begin{equation*}
        L(\piu) := \max\left\{\frac{\max\{\lvert a-m\rvert, \lvert b -m \rvert\}}{\piu(1-\piu)}, \frac{b-a}{\piu}\right\},
    \end{equation*}
    and $\epsilon_n := \frac{1}{n} \sum_{j=1}^{n} \frac{1}{j}$. Moreover, set $\vh := \frac{1}{n(n-1)}\sum_{i<j} ((\tauh^m_{i,n}-c_i)-(\tauh^m_{j,n}-c_j))^2$, where
    \begin{equation*}
        c_i := \left(\max\left\{\frac{b-m}{\pi_i},\frac{m-a}{1-\pi_i}\right\}+\min\left\{\frac{a-m}{\pi_i},\frac{m-b}{1-\pi_i}\right\}\right)\Big/\,2.
    \end{equation*}
    Then, the $1-\alpha$ confidence intervals for~$\tau_n$ based on \citet{maurer_empirical_2009}'s concentration inequality is given by
    \begin{equation*}
        \left[\tauh^m_n \pm \Bigg(\sqrt{\frac{2\vh\log(3/\alpha)}{n}} + L(\piu)\frac{7\log(3/\alpha)}{3(n-1)}\Bigg)\right].
    \end{equation*}
    The $1-\alpha$ confidence intervals based on our Theorem~\ref{thm:mine-bernstein} are given by
    \begin{align*}
        \Bigg[\tauh^m_n &\pm \Bigg(\sqrt{\frac{2(1+\epsilon_{n-1})\vh\log(2/\alpha)}{n-1}} + L(\piu)\frac{\log(2/\alpha)}{n-1}\Bigg)\Bigg],\\
        \Bigg[\tauh^m_n &\pm (1+\epsilon_{n-1})\frac{\vh}{L(\piu)}\, h_1^{-1}\!\left(\frac{L(\piu)^2\log(2/\delta)}{(1+\epsilon_{n-1})(n-1)\vh}\right)\Bigg].
    \end{align*}
\end{proposition}
\begin{proof}
    The random variable $\tauh^m_{n,i}$ assumes the following two values
    \begin{equation*}
        \tauh^m_{n,i} \in \left\{\frac{Y_i(1)-m}{\pi_i}, - \frac{Y_i(0)-m}{1-\pi_i}\right\}
    \end{equation*}
    and is therefore contained in the interval
    \begin{equation*}
        \tauh^m_{n,i} \in [l_i, u_i] := \left[\min\left\{\frac{a-m}{\pi_i}, \frac{m-b}{1-\pi_i}\right\}, \max\left\{\frac{b-m}{\pi_i}, \frac{m-a}{1-\pi_i}\right\}\right].
    \end{equation*}
    Its midpoint is given by $c_i$ and, setting $B:=\max\{\lvert a-m\rvert, \lvert b -m \rvert\}$, we compute an upper bound for its length:
    \begin{align*}
        u_i-l_i &= \max\left\{\frac{b-m}{\pi_i}, \frac{m-a}{1-\pi_i}\right\} + \max\left\{\frac{m-a}{\pi_i}, \frac{b-m}{1-\pi_i}\right\}\\
        &= \max\left\{\frac{b-m}{\pi_i(1-\pi_i)}, \frac{m-a}{\pi_i(1-\pi_i)}, \frac{b-a}{\pi_i}, \frac{b-a}{1-\pi_i}\right\}\\
        &\leq \max\left\{\frac{B}{\piu(1-\piu)}, \frac{b-a}{\piu}\right\}=:L(\piu)
    \end{align*}
    For the midpoint-differenced estimator, $B$ equals $(b-a)/2$ and $L(\piu)$ simplifies accordingly:
    \begin{equation*}
        L(\piu) = \frac{b-a}{\piu} \max\left\{\frac{1}{2(1-\piu)},1\right\} = \frac{b-a}{\piu},
    \end{equation*}
    since $1-\piu \geq 1/2$. The confidence intervals now directly follow by applying Theorems~\ref{thm:mp} and~\ref{thm:mine-bernstein}.    
\end{proof}

\begin{remark}
    Comparing to the derivation of the Hoeffding inequality, we notice that the upper bound on $u_i-l_i$ is tighter than in the proof above, cf.\ proof of Proposition~\ref{prop:bernoulli-hoeffding}. This is a remnant of the fact that we additionally need to compute the midpoint of the support for the variance estimator $\vh$ for Bernstein-type inequalities. This is why we choose a potentially larger interval $[u_i, l_i]$ whose endpoints are fully known.
\end{remark}

\section{Proofs of Section~\ref{sec:cre}}
\subsection{Concentration of Neyman's Variance Estimator}\label{app:proof-neyman}

We divide the proof of Theorem~\ref{thm:neyman-conc} into several steps. First, we derive a self-bounding property of Neyman's variance estimator, then we state a modified log-Sobolev inequality from the literature and finally use it to complete the proof.

\begin{lemma}[Self-boundedness of $\vhn$]\label{lem:self-bounded}
	Define the swap operator $\tau_{ij}\colon \{0,1\}^n \to \{0,1\}^n$ as
    		$\tau_{ij}(z) = (z_1,\ldots, z_{i-1},z_j,z_{i+1},\ldots, z_{j-1}, z_i, z_{j+1},\ldots,z_n)$,
	where $i,j\in\{1,\ldots,n\}$. Neyman's variance estimator is self-bounded in the sense that
	\begin{equation*}
		\sum_{i,j=1}^n \big(\vhn(Z)-\vhn(\tau_{ij}Z)\big)_+^2\, \leq\, 2 (b-a)^2\,c(n_0,n_1)\, \vhn(Z),
	\end{equation*}
	where the constant is given by
	\begin{equation*}
		c(n_0,n_1) := \frac{n\,(n_0^2+n_1^2)^2}{n_0n_1^3(n_0-1) + n_1 n_0^3(n_1-1)}.
	\end{equation*}
\end{lemma}
\begin{proof}
	We abbreviate the left-hand side of the inequality as $L(Z) := \sum_{i,j=1}^n \big(\vhn(Z)-\vhn(\tau_{ij}Z)\big)_+^2$. First, we notice that only index pairs $(i,j)$ with $Z \neq \tau_{ij} Z$ contribute to the sum; hence,
	\begin{align*}
		L(Z) &= \sum_{i,j=1}^n \big(\vhn(Z)-\vhn(\tau_{ij}Z)\big)_+^2 \one \{Z \neq \tau_{ij} Z\}\\
		&= \sum_{i,j=1}^n \big(\vhn(Z)-\vhn(\tau_{ij}Z)\big)_+^2 \Big(Z_i (1-Z_j) + (1-Z_i)Z_j\Big)\\
		&= 2\sum_{i,j=1}^n Z_i (1-Z_j) \big(\vhn(Z)-\vhn(\tau_{ij}Z)\big)_+^2,
	\end{align*}
	where the last step follows from the symmetry in the summation. Next, we separate the Neyman variance estimator into its two summands which are given by
    \begin{equation*}
        \vhn = \vhn^1+\vhn^0 := \frac{n}{n_1(n_1-1)}\sum_{i \in I_1} (Y_i(1)-\ybh(1))^2 + \frac{n}{n_0(n_0-1)}\sum_{i \in I_0} (Y_i(0)-\ybh(0))^2,
    \end{equation*}
    where $I_k := \{i\colon Z_i=k\}$ for $k\in \{0,1\}$. We fix indices $i\in I_1$ and $j \in I_0$ and consider the first summand under $Z$ and the swapped treatment $\tau_{ij}Z$:
    \begin{align*}
        \vhn^1(Z) &= \frac{n}{n_1(n_1-1)}\sum_{k \in I_1} (Y_k(1)-\ybh(1))^2,\\
        \vhn^1(\tau_{ij}Z) &= \frac{n}{n_1(n_1-1)}\sum_{k \in I_1 \setminus \{i\} \cup \{j\}} (Y_k(1)-\ybh_{-i+j}(1))^2.
    \end{align*}
    Here, $\ybh(1)$ denotes the average of the $Y(1)$-potential outcomes over the indices $I_1$ and $\ybh_{-i+j}(1)$ is the average over the indices $I_1\setminus\{i\} \cup \{j\}$. Applying Lemma~\ref{lem:emp-mean-perturb} with $I = I_1$ and $l=i$ to the first expression and with $I=I_1 \setminus \{i\} \cup \{j\}$ and $l=j$ to the second, we obtain
    \begin{align*}
        \vhn^1(Z) &= \frac{n}{n_1(n_1-1)}\left[\sum_{l' \in I_1 \setminus\{i\}} (Y_{l'}(1)-\ybh_{-i}(1))^2 + \frac{n_1-1}{n_1}(Y_i(1)-\ybh_{-i}(1))^2\right],\\
        \vhn^1(\tau_{ij}Z) &= \frac{n}{n_1(n_1-1)} \left[\sum_{l' \in I_1 \setminus \{i\} \cup \{j\}} (Y_{l'}(1)-\ybh_{-i}(1))^2 + \frac{n_1-1}{n_1} (Y_j(1)-\ybh_{-i}(1))^2\right].
    \end{align*}
    We take the difference and get
    \begin{equation*}
        \vhn^1(Z) - \vhn^1(\tau_{ij}Z) = \frac{n}{n_1^2}\Big[(Y_i(1)-\ybh_{-i}(1))^2-(Y_j(1)-\ybh_{-i}(1))^2\Big].
    \end{equation*}
    Following the same steps for $\vhn^0$, we obtain
    \begin{equation*}
        \vhn^0(Z) - \vhn^0(\tau_{ij}Z) = \frac{n}{n_0^2}\Big[(Y_j(0)-\ybh_{-j}(0))^2-(Y_i(0)-\ybh_{-j}(0))^2\Big].
    \end{equation*}

    We can now return to $L(Z)$, plug in our findings above and estimate the resulting expression using the monotonicity of the positive part:
    \begin{align*}
        L(Z) &= \begin{aligned}[t]
            2 n^2 \sum_{i,j=1}^n Z_i (1-Z_j) \bigg(&\frac{(Y_i(1)-\ybh_{-i}(1))^2-(Y_j(1)-\ybh_{-i}(1))^2}{n_1^2}\\
            &+ \frac{(Y_j(0)-\ybh_{-j}(0))^2-(Y_i(0)-\ybh_{-j}(0))^2}{n_0^2}\bigg)_+^2
        \end{aligned}\\[1ex]
        &\leq 2 n^2 \sum_{i,j=1}^n Z_i (1-Z_j) \left(\frac{(Y_i(1)-\ybh_{-i}(1))^2}{n_1^2}
            + \frac{(Y_j(0)-\ybh_{-j}(0))^2}{n_0^2}\right)^2.\\
    \end{align*}
     
     Next, we expand the square and find an upper bound using that $\lvert Y_i(k) - \yb_{-i}(k) \rvert \leq b-a$ for all $i \in \{1,\ldots,n\}$ and $k\in \{0,1\}$. Moreover, we can replace the ``leave-one-out'' averages with their usual counter-parts via $\ybh_{-i}(1)= \frac{n_1\ybh(1)-Y_i(1)}{n_1-1}$ and $\ybh_{-j}(0)= \frac{n_0\ybh(0)-Y_j(0)}{n_0-1}$ and subsequently gather terms:
     {\allowdisplaybreaks
    \begin{align*}
        L(Z) &= \begin{aligned}[t]
        \frac{2 n_0 n^2}{n_1^4} &\sum_{i \in I_1} (Y_i(1)-\ybh_{-i}(1))^4 + \frac{2 n_1 n^2}{n_0^4} \sum_{j \in I_0} (Y_j(0)-\ybh_{-j}(0))^4 \\
        &+ \frac{4n^2}{n_1^2 n_0^2} \sum_{i \in I_1} \sum_{j \in I_0} (Y_i(1)-\ybh_{-i}(1))^2 (Y_j(0)-\ybh_{-j}(0))^2
        \end{aligned}\\
        &\leq\begin{aligned}[t]
        2 (b-a)^2 n^2\Bigg[
        \frac{n_0}{n_1^4} &\sum_{i \in I_1} (Y_i(1)-\ybh_{-i}(1))^2 + \frac{n_1}{n_0^4} \sum_{j \in I_0} (Y_j(0)-\ybh_{-j}(0))^2 \\
        &+ \frac{2}{n_1^2 n_0^2} \sum_{i \in I_1} \sum_{j \in I_0} \lambda (Y_i(1)-\ybh_{-i}(1))^2 + (1-\lambda) (Y_j(0)-\ybh_{-j}(0))^2\Bigg]
        \end{aligned}\\
        &=\begin{aligned}[t]
        2 (b-a)^2 n^2\Bigg[
        &\frac{n_0}{n_1^2(n_1-1)^2} \sum_{i \in I_1} (Y_i(1)-\ybh(1))^2 + \frac{n_1}{n_0^2(n_0-1)^2} \sum_{j \in I_0} (Y_j(0)-\ybh(0))^2 \\
        & + \frac{2 \lambda}{(n_1-1)^2n_0} \sum_{i \in I_1} Y_i(1)-\ybh(1))^2 + \frac{2 (1-\lambda)}{(n_0-1)^2n_1} \sum_{j \in I_0} Y_j(0)-\ybh(0))^2\Bigg]
        \end{aligned}\\
        &=2(b-a)^2 n \left[\left(\frac{n_0}{n_1(n_1-1)} + \frac{2\lambda n_1}{n_0 (n_1-1)}\right) \vhn^1 + \left(\frac{n_1}{n_0(n_0-1)}+\frac{2(1-\lambda)n_0}{n_1(n_0-1)}\right) \vhn^0\right].
    \end{align*}
    }
    Here, we have introduced the hyper-parameter $\lambda \in \R$ which trades off the contributions of the two diagonal terms to the bound on the cross term. We can choose it so that the two constants in front of $\vhn^1$ and $\vhn^0$ agree. This is achieved at
    \begin{equation*}
        \lambda = \frac{1}{2} \frac{n_1^2(n_1-1) - n_0^2(n_0-1) + 2n_0^2(n_1-1)}{n_0^2(n_1-1)+n_1^2(n_0-1)}.
    \end{equation*}
    Inserting this choice of $\lambda$ into the upper bound of $L(Z)$, we arrive at the result
    \begin{equation*}
        L(Z) \leq 2 (b-a)^2 \frac{n\,(n_0^2+n_1^2)^2}{n_0n_1^3(n_0-1) + n_1 n_0^3(n_1-1)} (\vhn^1+\vhn^0).
    \end{equation*}
\end{proof}

\begin{theorem}[Modified log-Sobolev inequality]\label{thm:mod-log-sob}
Define the entropy of a non-negative random variable $X$ as $\ent[X] := \E[X\log(X)]-\E[X]\log(\E[X])$, let $g \colon \{0,1\}^n \to \R$ be a positive function and $Z \sim \mathrm{CRE}(n_0,n_1)$. Then, $g$ fulfills the modified log-Sobolev inequality
\begin{equation*}
    \ent[g(Z)] \leq\, \frac{1}{2n}\,  \E\left[\sum_{1\leq i < j \leq n} (g(Z)-g(\tau_{ij}Z))(\log g(Z)-\log g(\tau_{ij}Z))\right].
\end{equation*}
\end{theorem}
\begin{proof}
    The completely randomized treatment assignment random vector is uniformly distributed over the 2-multislice. \cite{salez_sharp_2021}[Lem.~1.3] proved that regardless of $(n_0,n_1)$ the constant in the modified log-Sobolev inequality is always upper bounded by $1$ which yields the result.
\end{proof}

\begin{proof}[Proof of Theorem~\ref{thm:neyman-conc}] Let $\lambda \in \R$. We invoke Theorem~\ref{thm:mod-log-sob} with the function $g = e^{\lambda \vhn}$ and obtain a modified log-Sobolev inequality for the Neyman variance estimator. We modify its right-hand side so that the summation extends over all index pairs $(i,j)$ and so that it involves the positive part in view of Lemma~\ref{lem:self-bounded}:
    \begin{align*}
        \ent[e^{\lambda \vhn(Z)}] &\leq \frac{1}{2n}\E\left[\sum_{1\leq i<j\leq n} (e^{\lambda \vhn(Z)}-e^{\lambda \vhn(\tau_{ij}Z)})\,(\lambda\vhn(Z)-\lambda\vhn(\tau_{ij}Z))\right]\\
        &= \frac{1}{4n}\E\left[\sum_{i,j=1}^n (e^{\lambda \vhn(Z)}-e^{\lambda \vhn(\tau_{ij}Z)})\,(\lambda\vhn(Z)-\lambda\vhn(\tau_{ij}Z))\right]\\
        &=\begin{aligned}[t]
            \frac{1}{4n} \sum_{i,j=1}^n &\E\left[(e^{\lambda \vhn(Z)}-e^{\lambda \vhn(\tau_{ij}Z)})\,(\lambda\vhn(Z)-\lambda\vhn(\tau_{ij}Z))_+\right]\\
            &-\E\left[(e^{\lambda \vhn(Z)}-e^{\lambda \vhn(\tau_{ij}Z)})\,(\lambda\vhn(Z)-\lambda\vhn(\tau_{ij}Z))_-\right]
        \end{aligned}\\
        &=\begin{aligned}[t]
            \frac{1}{4n} \sum_{i,j=1}^n &\E\left[(e^{\lambda \vhn(Z)}-e^{\lambda \vhn(\tau_{ij}Z)})\,(\lambda\vhn(Z)-\lambda\vhn(\tau_{ij}Z))_+\right]\\
            &-\E\left[(e^{\lambda \vhn(\tau_{ij}Z)}-e^{\lambda \vhn(Z)})\,(\lambda\vhn(\tau_{ij}Z)-\lambda\vhn(Z))_-\right]
        \end{aligned}\\
        &= \frac{1}{2n}\E\left[\sum_{i,j=1}^n (e^{\lambda \vhn(Z)}-e^{\lambda \vhn(\tau_{ij}Z)})\,(\lambda\vhn(Z)-\lambda\vhn(\tau_{ij}Z))_+\right]\\
        &\leq \frac{\lambda^2}{2n}\E\left[\sum_{i,j=1}^n e^{\lambda \vhn(Z)}(\vhn(Z)-\vhn(\tau_{ij}Z))_+^2\right]
    \end{align*}
    The first equality stems from the symmetry of the swap operator, i.e.\ $\tau_{ij} = \tau_{ji}$, and that the ``diagonal swaps'' do not contribute, i.e.\ $\tau_{ii}Z=Z$. We obtain the positive part since $Z$ is uniformly distributed over the 2-multislice and thus $\tau_{ij} Z$ and $Z$ have the same distribution. (In the literature on concentration inequalities for random walks, a similar argument is used when the generator of the random walk satisfies a detailed-balance condition, see \cite{adamczak_concentration_2021} for instance.) The last inequality follows from $(e^x-e^y)(x-y)_+ \leq e^x(x-y)_+^2$.

    Invoking Lemma~\ref{lem:self-bounded}, we find the following bound on the entropy
    \begin{equation*}
        \ent[e^{\lambda \vhn(Z)}] \leq \underbrace{\frac{(b-a)^2 c(n_0,n_1)}{n}}_{=:c}\,\lambda^2\E\left[e^{\lambda \vhn(Z)}\, \vhn(Z)\right].
    \end{equation*}
    To prove a concentration inequality for $\vhn^{1/2}$, we use Herbst's argument for self-bounding functions. 
    To keep the notation concise, we suppress the dependence of $\vhn$ on $Z$ in the following. Using the definition of the entropy, we get
    \begin{equation*}
        (\lambda-c\lambda^2) \E[\vhn e^{\lambda \vhn}] - \E[e^{\lambda \vhn}]\log\E[e^{\lambda \vhn}] \leq 0.
    \end{equation*}
    We can re-express this inequality in terms of the log (centred) moment generating function (log-mgf) and its derivative, which are given by
    \begin{align*}
        G(\lambda) &:= \log \E[e^{\lambda(\vhn-E[\vhn])}] = \log\E[e^{\lambda \vhn}]-\lambda\E[\vhn],\\
        G'(\lambda) &= \frac{\E\left[(\vhn-\E[\vhn])\, e^{\lambda \vhn}\right]}{\E\left[e^{\lambda \vhn}\right]} = \frac{\E[\vhn e^{\lambda \vhn}]}{\E[e^{\lambda \vhn}]}-\E[\vhn],
    \end{align*}
    respectively. Plugging these into the inequality, we obtain
    \begin{equation*}
        (\lambda-c\lambda^2)\, \Big(\E[e^{\lambda \vhn}]\,(G'(\lambda)+\E[\vhn])\Big) - \E[e^{\lambda \vhn}] \Big(G(\lambda) + \lambda \E[\vhn]\Big) \leq 0.
    \end{equation*}
    We further re-arrange this expressions by dividing by $\E[e^{\lambda \vhn}]$ and $(1- c\lambda)^2$ and get
    \begin{equation*}
        \frac{\lambda}{1-c\lambda}G'(\lambda) - \frac{G(\lambda)}{(1-c\lambda)^2} \leq \frac{\lambda^2c\,\E[\vhn]}{(1-c\lambda)^2}.
    \end{equation*}
    This differential inequality matches the form in Lemma~6.25 in \cite{boucheron_concentration_2013} with $f(\lambda) = \frac{\lambda}{1-c\lambda}$ and $g(\lambda) = c\E[\vhn]$ yielding
    \begin{equation*}
        G(\lambda) \leq \frac{\lambda^2c\,\E[\vhn]}{1-c\lambda},
    \end{equation*}
    for all $\lambda < \frac{1}{c}$. This relationship shows that $\vhn-\E[\vhn]$ is a sub-Gamma random variable on the right tail and can be used to construct a concentration inequality. Since we are interested in the left tail, however, we consider the log-mgf of $\E[\vhn]-\vhn$ denoted $\tilde{G}(\lambda)$. Since $G(-\lambda)=\tilde{G}(\lambda)$, we obtain
    \begin{equation*}
        \tilde{G}(\lambda) \leq \frac{\lambda^2 c \E[\vhn]}{1+ c\lambda}
    \end{equation*}
    for all $\lambda \geq 0$. We can use this inequality in the following Chernoff bound
    \begin{equation*}
        \PR\big(\E[\vhn]-\vhn \geq t\big) \leq \exp({\tilde{G}(\lambda)-\lambda t}) \leq \exp\left({\frac{\lambda^2 c \E[\vhn]}{1+ c\lambda} - \lambda t}\right),
    \end{equation*}
    where $t \in \R$. We can now solve the right-hand side for $\delta \in (0,1)$ and thus obtain that with probability $1-\delta$,
    \begin{alignat*}{3}
        &&\E[\vhn] &\leq \vhn + \frac{\log(1/\delta)}{\lambda} + \frac{\lambda c }{1+c\lambda}\E[\vhn],\\
        &\Leftrightarrow&\quad\E[\vhn] &\leq \left(\vhn + \frac{\log(1/\delta)}{\lambda}\right)(1+c\lambda) = \vhn + c\lambda \vhn + \frac{\log(1/\delta)}{\lambda} + c\log(1/\delta).
    \end{alignat*}
    Since we aim to get a concentration result for the square root of $\vhn$, we choose $\lambda$ so that the two terms in the middle match. This is achieved at $\lambda = \sqrt{\log(1/\delta)/(c\vhn)}$ and yields the inequality
    \begin{equation*}
        \E[\vhn] \leq \left(\vhn^{1/2} + \sqrt{c\log(1/\delta)}\right)^2.
    \end{equation*}
    Lastly, the Neyman variance estimator is conservative, i.e.\ $V \leq \E[\vhn]$, and we can plug in the definition of $c$. Thus, we obtain that with probability $1-\delta$,
    \begin{equation*}
        V^{1/2} \leq \E[\vhn]^{1/2} \leq \vhn^{1/2} + (b-a) \sqrt{\frac{c(n_0,n_1)}{n}\log(1/\delta)}.
    \end{equation*}
\end{proof}

\subsection{Confidence Intervals based on \cite{barber_hoeffding_2024}}\label{app:cre-bernstein-barber}
\begin{proof}[Proof of Proposition~\ref{prop:cre-bernstein}]
As we have seen in the proof of Proposition~\ref{prop:cre-hoeffding}, we can express the difference of the HT estimator and the SATE as follows
\begin{equation*}
    \tauh_n-\tau = \tauh_n^m-\tau = \sum_{i=1}^n\frac{1}{n} \left(\frac{Y_i(1)-\yb(1)}{\pi}+\frac{Y_i(0)-\yb(0)}{1-\pi}\right)(Z_i-\pi)=:\sum_{i=1}^n w_i X_i.
\end{equation*}
Here, we set $m=(1-\pi)\yb(1)+\pi \yb(0)$,
where $\yb(k):=\frac{1}{n}\sum_{i=1}^n Y(k)$ for $k\in \{0,1\}$. Since $X_i \in \{-\pi,1-\pi\}$, the exchangeable random variables $X_i$ are supported on an interval of length~$1$. Furthermore, $\xb = 0$ and the empirical variance of the $X_i$ is given by
\begin{equation*}
    \sigma_X^2 = \frac{1}{n}\sum_{i=1}^n (Z_i-\pi)^2 = \frac{1}{n} \left(n_1 \frac{n_0^2}{n^2}+n_0\frac{n_1^2}{n^2}\right) = \frac{n_0\,n_1}{n^2}.
\end{equation*}

Since the weights $w_i$ depend on unobserved potential outcomes, we need to find known or estimable upper bounds on their $L^\infty$- and $L^2$-norms to apply Theorem~\ref{thm:rina-bernstein}. First, the maximum value of the~$w_i$ can be estimated as follows
\begin{align*}
    \lVert w \rVert_\infty \leq \frac{1}{n}\left(\frac{1}{\pi}+\frac{1}{1-\pi}\right) \frac{n-1}{n} (b-a) = \frac{n-1}{n_0\,n_1}(b-a).
\end{align*}
Next, we re-write the squared $L^2$-norm of the weights:
\begin{align*}
    \lVert w \rVert_2^2 &= \frac{1}{n^2} \sum_{i=1}^n \frac{(Y_i(1)-\yb(1))^2}{\pi^2} + \frac{(Y_i(0)-\yb(0))^2}{(1-\pi)^2} + \frac{2(Y_i(1)-\yb(1))(Y_i(0)-\yb(0))}{\pi(1-\pi)}\\
    &= \frac{n-1}{n_1^2}S^2(1)+\frac{n-1}{n_0^2}S^2(0)+\frac{2(n-1)}{n_0\,n_1}S^2(1,0).
\end{align*}
Since we cannot consistently estimate this quantity due to the cross-term $S^2(1,0)$, we use an upper bound, see \cite[Lem.~4.1]{ding_first_2024} for instance:
\begin{equation*}
    \lVert w \rVert_2^2 \leq S^2(1)\left(\frac{n-1}{n_1^2}+\frac{n-1}{n_0\,n_1}\right) + S^2(0)\left(\frac{n-1}{n_0^2}+\frac{n-1}{n_0\,n_1}\right) = \frac{n-1}{n_0\,n_1} \E[\vhn].
\end{equation*}
We can now derive the explicit Bernstein and the implicit Bennett confidence interval.

    \textit{Bernstein Interval}
    We apply Theorem~\ref{thm:rina-bernstein} with $N=n$ and the bounds on the respective expressions above and obtain that with probability $1-2\delta$,
    \begin{multline*}
        \lvert \tauh_n-\tau \rvert
        \leq \sqrt{2
        \left(\frac{n_0\,n_1}{n^2}+4\epsilon_n'\right) \frac{n-1}{n_0\,n_1} \E[\vhn](1+\epsilon'_n)\log(1/\delta)}\\[-1.5ex] + \frac{b-a}{3} \frac{n-1}{n_0\,n_1}
        (1+\epsilon_n')\log(1/\delta).
    \end{multline*}
    Moreover, due to Theorem~\ref{thm:neyman-conc},
    with probability $1-\delta$:
    \begin{equation*}
 \E[\vhn]^{1/2} \leq \vhn^{1/2} +(b-a) \sqrt{\frac{c(n_0,n_1)\log(1/\delta)}{n}}.
    \end{equation*}
    Combining these two results with a union bound and gathering terms, we obtain
    \begin{align*}
        \lvert \tauh_n-\tau\rvert &\leq
        \begin{aligned}[t]
            &\sqrt{2\left(\frac{n_0\,n_1}{n^2}+4\epsilon_n'\right) \frac{n(n-1)}{n_0\,n_1}(1+\epsilon_n')\frac{\vhn}{n}
        \log(1/\delta)}\\
        &\hspace{-1cm}+\frac{b-a}{n}\log(1/\delta)\left[\sqrt{2\,c(n_0,n_1)\left(\frac{n_0\,n_1}{n^2}+4\epsilon_n'\right) \frac{n(n-1)}{n_0\,n_1}(1+\epsilon_n')}+\frac{1+\epsilon_n'}{3}\frac{n(n-1)}{n_0\,n_1}\right]
        \end{aligned}\\
        &=
        \begin{aligned}[t]
            &\sqrt{\frac{2(1+\epsilon_n'')(1+\epsilon_n')\vhn\log(1/\delta)}{n}}\\
        &\hspace{1.7cm}+ \frac{b-a}{n}\log(1/\delta)\left[\frac{1+\epsilon_n'}{3}\frac{n(n-1)}{n_0\,n_1}+\sqrt{2c(n_0,n_1)(1+\epsilon_n'')(1+\epsilon_n')}\right]
        \end{aligned}
    \end{align*}
    with probability $1-3\delta$. Setting $\delta = \alpha/3$ yields the Bernstein confidence interval.

    \textit{Bennett Interval}
        We invoke the second part of Theorem~\ref{thm:rina-bernstein} with the $N=n$ and the expressions above and find that with probability $1-2\delta$,
    \begin{align*}
        \lvert \tauh_n-\tau\rvert
        &\leq \frac{\left(\frac{n_0\,n_1}{n^2}+4\epsilon_n'\right) \frac{n-1}{n_0\,n_1} \E[\vhn]}{\frac{n-1}{n_0\,n_1}(b-a)}\, h_2^{-1}\left(\frac{(1+\epsilon_n')\frac{(n-1)^2}{n_0^2n_1^2}(b-a)^2\log(1/\delta)}{\left(\frac{n_0\,n_1}{n^2}+4\epsilon_n'\right) \frac{n-1}{n_0\,n_1} \E[\vhn]}\right)\\
        &= \frac{(1+\epsilon_n'')\,n_0\,n_1\,\E[\vhn]}{n(n-1)(b-a)}\,h_2^{-1}\left(\frac{(1+\epsilon_n')\,n(n-1)^2\,(b-a)^2 \log(1/\delta)}{(1+\epsilon_n'')\,n_0^2n_1^2\,\E[\vhn]}\right).
    \end{align*}
    Moreover, due to Theorem~\ref{thm:neyman-conc},
    \begin{equation*}
        \tilde{V} = \left(\vhn^{1/2}+(b-a)\sqrt{\frac{c(n_0,n_1)\log(1/\delta)}{n}}\right)^2
    \end{equation*}
    is larger than $\E[\vhn]$ with probability $1-\delta$.
    Combining these two results with a union bound and setting $\delta = \alpha/3$, we conclude the proof of the Bennett confidence interval.
\end{proof}

\subsection{Confidence Intervals based on \cite{bardenet_concentration_2015}}\label{app:cre-bernstein-bardenet}

We can also derive the empirical, explicit Bernstein and implicit Bennett interval using a concentration inequality from \citet{bardenet_concentration_2015} re-stated in Theorem~\ref{thm:bm}. As we can see in Figure~\ref{fig:cre-bernstein-comp}, these are wider than the ones based on \citet{barber_hoeffding_2024}'s concentration inequality, cf.\ Theorem~\ref{thm:rina-bernstein}.

\begin{figure}[htbp]
    \centering
    \includegraphics[scale=0.65]{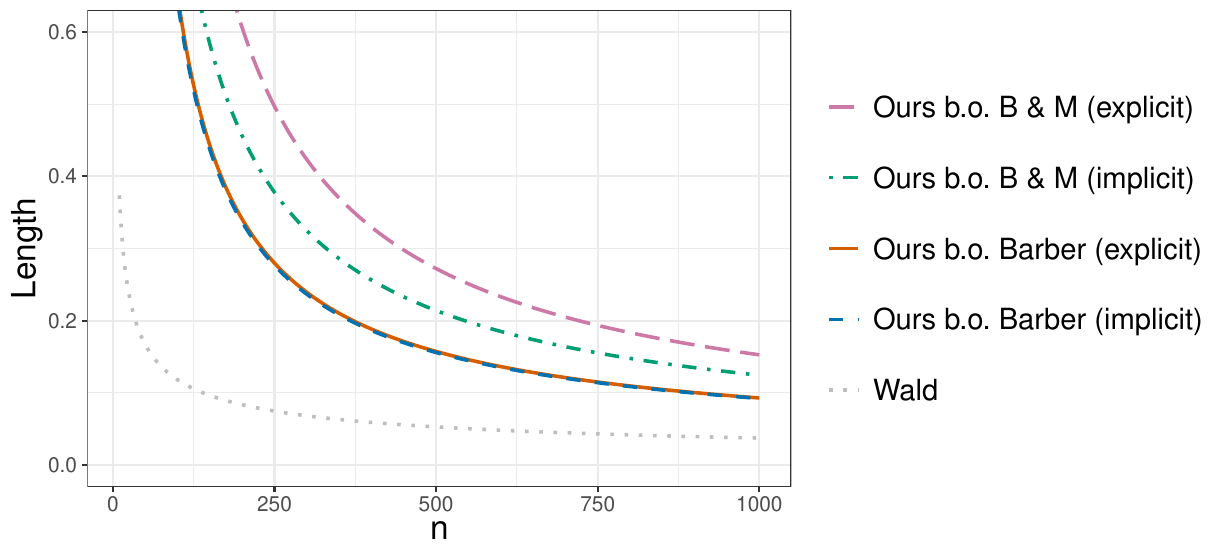}
    \caption{Length of confidence intervals in a completely randomized experiment as a function of $n$. We compare our intervals from Proposition~\ref{prop:cre-bernstein} based on (b.o.) \citet{barber_hoeffding_2024}'s concentration inequality with the intervals from Proposition~\ref{prop:cre-bernstein-bm} based on \citet{bardenet_concentration_2015}'s concentration inequality. We set $\alpha = 0.05, \pi=1/2$ and $[a,b] = [0,1]$. Moreover, we use $\vh = \sigma^2_{5,5}/(\pi(1-\pi))$, where $\sigma^2_{5,5}$ is the variance of a $\mathrm{Beta}(5,5)$-distribution.}
    \label{fig:cre-bernstein-comp}
\end{figure}

\begin{proposition}\label{prop:cre-bernstein-bm}
    Suppose $Z \sim \cre(n_0,n_1)$, where $n_0 \geq n_1 \geq 2$, and let $\alpha \in (0,1)$. Let $c(n_0,n_1)$ be defined as in Theorem~\ref{thm:neyman-conc} and set $\epsilon_n''' := \frac{1}{n_0}-\frac{1}{n}-\frac{1}{nn_0}$. Then, an explicit Bernstein $1-\alpha$ confidence interval for $\tau_n$ is given by
    \begin{equation*}
        \left[\tauh_n \pm \left(\sqrt{\frac{2(1+\epsilon'''_n) \vhn\log(4/\alpha)}{n}} + \frac{C (b-a)}{n}\log(4/\alpha)\right)\right],
    \end{equation*}
    where
    \begin{equation*}
        C = \sqrt{2(1+\epsilon_n''')c(n_0,n_1)}+ \frac{n(n-1)}{n_0n_1}\left(\frac{4}{3}+\sqrt{\frac{n_1(n_1-1)}{n(n_0+1)}}\right).
    \end{equation*}
    The tighter Bennett $1-\alpha$ confidence interval takes the form
    \begin{equation*}
        \left[\tauh_n\pm \frac{n_1^2}{(n-1)n_0}  \frac{\tilde{V}}{2(b-a)} h^{-1}_2\left(\frac{(n-1)^2}{n_1^2}\frac{4(b-a)^2\log(4/\alpha)}{n_1\,\tilde{V}}\right)\right],
    \end{equation*}
    where $h_2(u):=(1+u)\log(1+u)-u$ and
    \begin{align*}
        \tilde{V} = &(1+\epsilon_n''')\,\frac{n_1}{n}\,\left(\vhn^{1/2}+ (b-a)\sqrt{c(n_0,n_1)\log(4/\alpha)}\right)^2\\
        &+ \sqrt{\frac{(n-1)^3(n_1-1)n_1}{n_0^3\,n^3}}\left(\vhn^{1/2} + (b-a)\sqrt{c(n_0,n_1)\log(4/\alpha)}\right)(b-a)\sqrt{2\log(4/\alpha)}.
    \end{align*}
\end{proposition}

\begin{proof}
    To find empirical concentration inequalities for the HT estimator that account for the variance, we use modified versions of the results that \cite{bardenet_concentration_2015} obtained for sampling without replacement. To this end, we manipulate $\tauh - \tau$ as follows
    \begin{align*}
        \tauh_n - \tau_n &= \frac{1}{n_1} \sum_{i=1}^n Z_i Y_i(1) - \frac{1}{n_0} \sum_{i=1}^n (1-Z_i) Y_i(0) - \frac{1}{n} \sum_{i=1}^n Y_i(1)-Y_i(0) \\
        &= \frac{1}{n_1} \sum_{i=1}^n \left(\underbrace{(Y_i(1)-\yb(1))+\frac{n_1}{n_0}(Y_i(0)-\yb(0))}_{=:x_i}\right) Z_i,
    \end{align*}
    where $\yb(k) = \frac{1}{n}\sum_{i=1}^n Y_i(k)$. This shows, that $\tauh-\tau$ can also be regarded as the average of~$n_1$ data points sampled without replacement from the population $(x_1,\ldots,x_n)$. Due to the definition of the $x_i$, their mean equals 0 -- $\mu = \frac{1}{n}\sum_{i=1}^n x_i =0$ -- and
    \begin{equation*}
        \lVert x \rVert_\infty \leq \frac{n-1}{n} (b-a) + \frac{n_1}{n_0}\frac{n-1}{n} (b-a) = \frac{n-1}{n_0} (b-a).
    \end{equation*}
    Moreover, the (finite population) variance is given by
    \begin{align*}
        \sigma^2 &= \frac{1}{n} \sum_{i=1}^n x_i^2\\
        &= \frac{1}{n} \sum_{i=1}^n (Y_i(1)-\yb(1))^2+\frac{n_1^2}{n_0^2} (Y_i(0)-\yb(0))^2 + 2\frac{n_1}{n_0}(Y_i(1)-\yb(1))(Y_i(0)-\yb(0))\\
        &= \frac{n-1}{n} S^2(1) + \frac{n_1^2}{n_0^2} \frac{n-1}{n}S^2(0) + 2\frac{n_1}{n_0}\frac{n-1}{n} S^2(1,0).
    \end{align*}
    Since $\sigma^2$ cannot be consistently estimated due to the cross-term $S^2(1,0)$, we use an upper bound, see \cite[Lem.~4.1]{ding_first_2024} for instance:
    \begin{equation}\label{eq:sigma-upper-bound}
        \sigma^2 
        \leq \frac{n-1}{n} \left(\frac{n}{n_0}S^2(1)+\frac{n \,n_1}{n_0^2}S^2(0)\right) = \frac{n-1}{n}\frac{n_1}{n_0} \E[\vhn].
    \end{equation}
    We are now in position to derive the explicit and implicit confidence intervals.

    \textit{Bernstein Interval} We apply Theorem~\ref{thm:bm} with $k=n_1$, $B=\frac{n-1}{n_0}(b-a)$ and obtain, that with probability $1-3\delta$
    \begin{equation*}
        \left\vert\frac{1}{n_1}\sum_{i=1}^n x_i Z_i\right\vert \leq \sigma \sqrt{\frac{n_0+1}{n\,n_1}2\log(1/\delta)}+\left(\frac{4}{3n_1} + \sqrt{\frac{n_1-1}{n\,n_1\,(n_0+1)}}\right) \frac{n-1}{n_0}(b-a)\log(1/\delta).
    \end{equation*}
    Moreover, due to Theorem~\ref{thm:neyman-conc}, with probability $1-\delta$:
    \begin{equation*}
        \sigma \leq \sqrt{\frac{(n-1)n_1}{n\,n_0}}\, \E[\vhn]^{1/2} \leq \sqrt{\frac{(n-1)n_1}{n\,n_0}}\, \left(\vhn^{1/2} +(b-a) \sqrt{\frac{c(n_0,n_1)\log(1/\delta)}{n}}\right).
    \end{equation*}
    Combining these two results with a union bound and gathering terms, we obtain
    \begin{align*}
        \left\vert\tauh_n - \tau_n\right\vert &\leq
        \begin{aligned}[t]
        &\sqrt{2\frac{n-1}{n}\frac{n_0+1}{n_0} \frac{\vhn}{n}\log(1/\delta)}\\
        &+ \frac{b-a}{n}\log(1/\delta)\left[\sqrt{2\frac{n-1}{n}\frac{n_0+1}{n_0}\,c(n_0,n_1)} + \left(\frac{4n}{3n_1} + \sqrt{\frac{n(n_1-1)}{n_1\,(n_0+1)}}\right) \frac{n-1}{n_0}\right]
        \end{aligned}\\
        &=\begin{aligned}[t]
        &\sqrt{\frac{2(1+\epsilon''_n) \vhn\log(1/\delta)}{n}}\\
        &+
        \frac{b-a}{n}\log(1/\delta)\left[\sqrt{2(1+\epsilon'''_n)\,c(n_0,n_1)} + \left(\frac{4}{3} + \sqrt{\frac{n_1(n_1-1)}{n\,(n_0+1)}}\right) \frac{n(n-1)}{n_0n_1}\right],
        \end{aligned}
    \end{align*}
    with probability $1-4\delta$. Setting $\delta = \alpha/4$, we find the Bernstein-type confidence interval.

    \textit{Bennett Interval} For the tighter confidence interval, we invoke the second concentration inequality in Theorem~\ref{thm:bm} which in our setting yields
    \begin{equation*}
        \left\vert\tauh_n-\tau_n\right\vert \leq \frac{v}{2(b-a)}\frac{n_0}{n-1}\, h^{-1}_2\!\!\left(\frac{4(b-a)^2}{v}\frac{(n-1)^2}{n_0^2}\frac{\log(1/\delta)}{n_1}\right),
    \end{equation*}
    with probability $1-3\delta$. In the formula of $v$, we replace $\sigma^2$ with the upper bound~\eqref{eq:sigma-upper-bound} and obtain
    \begin{align*}
        v &=  \frac{n_1^2}{n_0^2}\Bigg[\frac{n_0+1}{n}\frac{n-1}{n}\frac{n_1}{n_0}\E[\vhn]+ \sqrt{\frac{n-1}{n}\frac{n_1}{n_0}\E[\vhn]} \frac{(n-1)(n_1-1)}{n_0\,n}(b-a) \sqrt{\frac{2\log(4/\delta)}{n_1-1}} \Bigg]\\
        &= \frac{n_1^2}{n_0^2}\left[(1+\epsilon'''_n)\,\frac{n_1}{n}\,\E[\vhn] + \sqrt{\frac{(n-1)^3(n_1-1)n_1}{n_0^3\,n^3}}\E[\vhn]^{1/2}(b-a)\sqrt{2\log(4/\delta)}\right].
    \end{align*}
    Lastly, we employ the concentration inequality in Theorem~\ref{thm:neyman-conc} for the oracle $\E[\vhn]^{1/2}$-term. Combining the results via a union bound, we get that with probability $1-4\delta$,
     \begin{equation*}
        \lvert\tauh_n - \tau_n \rvert \leq \frac{n_1^2}{(n-1)n_0}  \frac{\tilde{V}}{2(b-a)} h^{-1}_2\left(\frac{(n-1)^2}{n_1^2}\frac{4(b-a)^2\log(1/\delta)}{n_1\,\tilde{V}}\right),
    \end{equation*}
    where
     \begin{align*}
        \tilde{V} = &(1+\epsilon_n''')\,\frac{n_1}{n}\,\left(\vhn^{1/2}+ (b-a)\sqrt{c(n_0,n_1)\log(1/\delta)}\right)^2\\
        &+ \sqrt{\frac{(n-1)^3(n_1-1)n_1}{n_0^3\,n^3}}\left(\vhn^{1/2} + (b-a)\sqrt{c(n_0,n_1)\log(1/\delta)}\right)(b-a)\sqrt{2\log(1/\delta)}.
    \end{align*}
    Setting $\delta = \alpha/4$ concludes the proof.
\end{proof}

\subsection{Stratified Randomized Experiment}\label{app:scre}

\begin{proof}[Proof of Proposition~\ref{prop:scre}]
    Throughout this proof, we tacitly assume $j\in \{1,\ldots,J\}$.

    \textit{Hoeffding Interval} In the proof of Proposition~\ref{prop:cre-hoeffding}, we have shown that
    \begin{equation*}
        \E\left[\exp\left(\lambda_j (\tauh_{n_j}^j - \E[\tauh_{n_j}^j])\right)\right] \leq
        \exp\left(\frac{\lambda_j^2}{8}\left(\frac{b-a}{2}\right)^2 \frac{n_j^3}{n_{j,0}^2\,n_{j,1}^2} (1+\epsilon_{n_j}')\right),
    \end{equation*}
    for all $j$ and $\lambda_j \in \R$. Let now $\lambda \in \R$ and set $\lambda_j = \lambda s_j$. Due to independence of the $\tauh_{n_j}^j$, we obtain
    \begin{align*}
        \E\Big[\exp\left(\lambda (\tauh_n-\tau_n)\right)\Big] &=
        \E\left[\exp\left(\lambda \sum_{j=1}^J s_j (\tauh_{n_j}^j - \tau_n)\right)\right] \\
        &\leq \exp\left(
        \frac{\lambda^2}{8} \left(\frac{b-a}{2}\right)^2 \sum_{j=1}^J s_j^2 \frac{n_j^3}{n_{j,0}^2\,n_{j,1}^2} (1+\epsilon_{n_j}')\right).
    \end{align*}
    Inverting this bound on the moment generating function analogously to Theorem~\ref{prop:rina-hoeffding}, we see that with probability $1-\alpha$,
    \begin{equation*}
        \left\vert \tauh_n-\tau_n \right\vert \leq \frac{b-a}{2} \sqrt{\sum_{j=1}^J \frac{n_j^5}{n\, n_{j,0}^2\,n_{j,1}^2}(1+\epsilon_{n_j}')}\sqrt{\frac{\log(2/\alpha)}{2n}}.
    \end{equation*}

    \textit{Bernstein Interval} To generalize the confidence interval to SCREs, we separately extend Theorem~\ref{thm:neyman-conc} and the underlying \emph{oracle} Bernstein concentration inequality and combine them via a union bound.

    First, we consider $\vhns$. From the proof of Theorem~\ref{thm:neyman-conc}, we know that
    \begin{equation*}
        \E\left[\exp\left(\lambda_j\,(\E[\vhn^j]-\vhn^j)\right)\right] \leq
        \frac{\lambda_j^2c_j\E[\vhn^j]}{1+c_j\lambda_j},
    \end{equation*}
    for each each $j$ and all $\lambda_j \geq 0$, where $c_j := (b-a)^2c(n_{j,0},n_{j,1})/n_j$. Let $\lambda \in\R$ and set $\lambda_j = s_j^2 \lambda$. Then, we can again use independence between the strata to bound the moment generating function:
    \begin{multline*}
        \E\left[\exp\left(\lambda (\E[\vhns]-\vhns)\right)\right] = 
        \E\left[\exp\left(\lambda\sum_{j=1}^J s_j^2\, (\E[\vhn^j]-\vhn^j)\right)\right]\\
        \leq \exp\left(\sum_{j=1}^J \frac{\lambda^2s_j^2\,c_js_j^2 \E[\vhn^j]}{1+\lambda c_j s_j^2}\right)
        \leq \exp\Bigg(\frac{\lambda^2 c^*}{1+\lambda c^*} \underbrace{\sum_{j=1}^J s_j^2 \E[\vhn^j]}_{\E[\vhns]}\Bigg),
    \end{multline*}
    where $c^* := \max_{j} c_j s_j^2$. Tracing the remainder of the proof of Theorem~\ref{thm:neyman-conc} with $c$ replaced by $c^*$, we obtain that
    \begin{equation}\label{eq:neyman-scre}
        \E[\vhns]^{1/2} \leq (\vhns)^{1/2} + (b-a) \max_{j} \left\{\frac{n_j}{n}\sqrt{c(n_{j,0},n_{j,1})}\right\} \sqrt{\frac{\log(1/\delta)}{n}},
    \end{equation}
    with probability $1-\delta$ for any $\delta \in (0,1)$.

    Second, we turn to the oracle Bernstein concentration inequality. Due to Theorem~\ref{thm:rina-bernstein}, we know that
    \begin{equation*}
        \E\left[\exp\left(\lambda_j (\tauh^j_{n_j}-\E[\tauh^j_{n_j}])\right)\right] \leq 
        \exp\left(\frac{\lambda^2_j (1+\epsilon_{n_j}') \Big(\frac{n_{j,0}\,n_{j,1}}{n_j^2}+4\epsilon_{n_j}'\Big) \frac{n_j-1}{n_{j,0}\,n_{j,1}} \E[\vhn^j]}{2\Big(1-\frac{\lvert \lambda_j\rvert}{3} \frac{n_j-1}{n_{j,0}\,n_{j,1}} (b-a)(1+\epsilon_{n_j}')
        \Big)}\right)
    \end{equation*}
    for all $j$ and $\lambda_j$ satisfying $\lvert \lambda_j\rvert < \frac{3\,n_{j,0}n_{j,1}}{(n_j-1)(b-a)(1+\epsilon_{n_j}')}$.
    Let $\lambda \in \R$ such that
    \begin{equation*}
        \lvert \lambda \rvert < \frac{3}{(b-a)\max_{j} \left\{(1+\epsilon_{n_j}')(n_j-1)/(n_{j,0}\,n_{j,1})\right\}},
    \end{equation*}
    set $\lambda_j = \lambda s_j$ and recall the definiton of $\epsilon_{n_j}'' := 4\epsilon_{n_j}'\frac{n_j(n_j-1)}{n_{j,0}\,n_{j,1}}-\frac{1}{n_j}$. We again use independence between strata to combine the moment generating functions of the individual strata and then find an upper bound:
    \begin{align*}
        \E\left[\exp\left(\lambda (\tauh_n-\tau_n)\right)\right] &= \E\left[\exp\left(\lambda \sum_{j=1}^J s_j(\tauh_{n_j}^j-\tau_n)\right)\right]\\
        &\leq \exp\left(\lambda^2\sum_{j=1}^J \frac{s_j^2(1+\epsilon_{n_j}') \frac{1+\epsilon_{n_j}''}{n_j} \E[\vhn^j]}{2\Big(1-\frac{\lvert \lambda\rvert s_j}{3} \frac{n_j-1}{n_{j,0}\,n_{j,1}} (b-a)(1+\epsilon_{n_j}')
        \Big)}\right)\\
        &\leq \exp\left(\frac{\lambda^2 \max_{j}\left\{(1+\epsilon_{n_j}')(1+\epsilon_{n_j}'')/n_j\right\} \E[\vhns]}{2\Big(1-\frac{\lvert \lambda \rvert}{3}(b-a) \max_{j} \left\{(1+\epsilon_{n_j}')(n_j-1)/(n_{j,0}\,n_{j,1})\right\}\Big)}\right)
    \end{align*}
    Inverting this bound similarly to Theorem~\ref{thm:rina-bernstein}, we obtain that with probability $1-2\delta$,
    \begin{multline}\label{eq:or-bernstein-scre}
        \lvert \tauh_n -\tau_n\rvert \leq \sqrt{2 \max_j\left\{\frac{(1+\epsilon_{n_j}')(1+\epsilon_{n_j}'')}{n_j}\right\} \E[\vhns] \log(2/\delta)}\\
        + \frac{b-a}{3} \max_j \left\{\frac{(1+\epsilon_{n_j}')(n_j-1)}{n_{j,0}\,n_{j,1}}\right\} \log(2/\delta).
    \end{multline}
    Lastly, we combine the inequalities~\eqref{eq:neyman-scre} and~\eqref{eq:or-bernstein-scre} via union bound, set $\delta = \alpha/3$ and gather terms. Thus, we arrive at
    \begin{multline*}
        \lvert \tauh_n-\tau \rvert \leq \sqrt{2 \max_j\left\{\frac{(1+\epsilon_{n_j}')(1+\epsilon_{n_j}'')}{n_j}\right\} \vhns \log(3/\alpha)}\\
        +\frac{b-a}{n}\log(3/\alpha) \Bigg[\frac{1}{3} \max_j \left\{\frac{(1+\epsilon_{n_j}')n(n_j-1)}{n_{j,0}\,n_{j,1}}\right\}\\
        + \sqrt{2\max_j\left\{\frac{(1+\epsilon_{n_j}')(1+\epsilon_{n_j}'')\,n}{n_j}\right\}} \,\,\max_{j} \left\{\frac{n_j}{n}\sqrt{c(n_{j,0},n_{j,1})}\right\}\Bigg],
    \end{multline*}
    with probability $1-\alpha$.    
\end{proof}

\section{Algebraic Identities}\label{app:lin-alg}

\begin{lemma}\label{lem:perm-avg}
    For any matrix $B \in \R^{n\times n}$ and arbitrary permutation matrices $\Pi \in \R^{n\times n}$,
    \begin{equation*}
        \frac{1}{n!} \sum_{\Pi} \Pi^T B\, \Pi = \frac{\tr(B)}{n} I +\frac{\one^TB\,\one-\tr(B)}{n(n-1)}(\one\one^T-I),
    \end{equation*}
    where $I \in \R^{n\times n}$ is the identity matrix.
\end{lemma}
\begin{proof}
    Let $\one:=(1,\ldots,1)^T\in \R^n$, let $\mathcal{S}_n$ be the permutation group and $i,j\in\{1,\ldots,n\}$. For any permutation matrix $\Pi$ we denote its corresponding permutation $\pi \in \mathcal{S}_n$. Notably, we have $(\Pi^T B\, \Pi)_{i,j} = B_{\pi(i),\pi(j)}$. We now evaluate the on- and off-diagonal entries separately. For $i=j$, we find
    \begin{equation*}
        \frac{1}{n!}\sum_\Pi (\Pi^T B\, \Pi)_{i,i}  = \frac{1}{n!}\sum_{\pi\in \mathcal{S}_n} B_{\pi(i),\pi(i)}= \frac{1}{n!} \sum_{\substack{\pi\in\mathcal{S}_n,\\\pi(i)=l}}\,\, \sum_{l=1}^n B_{l,l} = \frac{(n-1)!}{n!} \tr(B) = \frac{\tr(B)}{n},
    \end{equation*}
    for $i\neq j$, we obtain
    \begin{align*}
        \frac{1}{n!}\sum_\Pi (\Pi^T B\, \Pi)_{i,j}  &= \frac{1}{n!}\sum_{\pi\in \mathcal{S}_n} B_{\pi(i),\pi(j)}= \frac{1}{n!} \sum_{\substack{\pi\in\mathcal{S}_n,\\\pi(i)=l,\\\pi(j)=l'}}\,\, \sum_{\substack{l,l'=1\\ l\neq l'}}^n B_{l,l'} = \frac{(n-2)!}{n!} (\one^T B\, \one - \tr(B))\\[-2ex]
        &=\frac{\one^TB\,\one-\tr(B)}{n(n-1)}.
    \end{align*}
    Assembling, the entry-wise results in a matrix concludes the proof.
\end{proof}

\begin{lemma}\label{lem:emp-mean-perturb}
    Let $I$ be a finite index set with $\lvert I \rvert=n$ and let the collection $(x_i)_{i\in I}$ be real-valued. Then, for any $l \in I$,
	\begin{equation*}
		\sum_{i \in I} (x_i - \xbs)^2 = \sum_{i \in I \setminus\{l\}} (x_i - \xbs_{-l})^2 + \frac{n-1}{n} (x_l - \xbs_{-l})^2,
	\end{equation*}
	where $\xbs:=\frac{1}{n}\sum_{i\in I} x_i$ and $\xbs_{-l}:=\frac{1}{n-1}\sum_{i \in I \setminus\{l\}}^n x_i$.
\end{lemma}

\begin{proof}
	By definition, $\xbs_{-l} = \frac{n \xbs - x_l}{n-1}$. Plugging this equation into the right-hand side of the statement and simplifying the resulting expression, we obtain
{\allowdisplaybreaks
    \begin{align*}
		 \sum_{i \in I \setminus\{l\}}^n (x_i - \xbs_{-l})^2 &+ \frac{n-1}{n} (x_l - \xbs_{-l})^2\\[-2ex]
		  	&=\sum_{I \setminus\{l\}} \left((x_i-\xbs)+\frac{x_l-\xbs}{n-1}\right)^2 + \frac{n}{n-1} (x_l - \xbs)^2\\
		 &= \sum_{I \setminus\{l\}} (x_i-\xbs)^2 + 2(x_i-\xbs)\frac{x_l-\xbs}{n-1}  + \frac{(x_l - \xbs)^2}{n-1} + \frac{n}{n-1} (x_l - \xbs)^2\\
		 &= \sum_{I \setminus\{l\}} (x_i-\xbs)^2 + 2\Big(n\xbs - x_l - (n-1)\xbs\Big) \frac{x_l-\xbs}{n-1} + \frac{n+1}{n-1} (x_l - \xbs)^2\\
		 &= \sum_{I \setminus\{l\}} (x_i-\xbs)^2 - \frac{2}{n-1}(x_l-\xbs)^2+ \frac{n+1}{n-1} (x_l - \xbs)^2\\
		 &= \sum_{i\in I} (x_i-\xbs)^2.
    \end{align*}
}
\end{proof}

\end{appendix}

\end{document}